\documentclass[10pt, oneside]{amsart}
\usepackage[centering,letterpaper,dvips,marginparwidth=2.5cm]{geometry}

\usepackage{graphicx}
\usepackage{amsfonts}
\usepackage{epsf}
\usepackage{amssymb}
\usepackage{amsmath}
\usepackage{amscd}
\usepackage{tikz}
\usepackage{pdfpages}
\usepackage{fancyhdr}
\usepackage{setspace}
\usepackage{hyperref}
\usepackage[all]{xy}
\usetikzlibrary{matrix}
\usepackage{verbatim}
\usepackage{enumerate}
\usepackage[normalem]{ulem}

\theoremstyle{theorem}
\newtheorem{theorem}{Theorem}[section]
\newtheorem{proposition}[theorem]{Proposition}
\newtheorem{lemma}[theorem]{Lemma}
\newtheorem{question}[theorem]{Question}

\newtheorem{thmx}{Theorem}
\newtheorem{corox}[thmx]{Corollary}

\makeatletter
\newtheorem*{rep@theorem}{\rep@title}
\newcommand{\newreptheorem}[2]{%
\newenvironment{rep#1}[1]{%
 \def\rep@title{#2 \ref{##1}}%
 \begin{rep@theorem}}%
 {\end{rep@theorem}}}
\makeatother

\newreptheorem{theorem}{Theorem}
\newreptheorem{lemma}{Lemma}
\newreptheorem{question}{Question}
\newreptheorem{corollary}{Corollary}
\newreptheorem{proposition}{Proposition}

\theoremstyle{definition}
\newtheorem{example}[theorem]{Example}

\newtheorem{definition}[theorem]{Definition}
\newtheorem{remark}[theorem]{Remark}

\newcommand{\Z}{\mathbb{Z}}

\newcommand{\RP}{\mathbb{RP}}
\newcommand{\CP}{\mathbb{CP}}

\newcommand{\id}{\text{id}}

\newcommand{\DD}{\mathfrak D}

\newcommand{\Cc}{\mathcal C}
\newcommand{\Dd}{\mathcal D}
\newcommand{\Ff}{\mathcal F}

\newcommand{\Jj}{\mathcal J}
\newcommand{\Kk}{\mathcal K}
\newcommand{\Ll}{\mathcal L}
\newcommand{\Mm}{\mathcal M}

\newcommand{\Ss}{\mathcal S}
\newcommand{\Tt}{\mathcal T}
\newcommand{\Uu}{\mathcal U}

\makeatletter
\newcommand*\wt[1]{\mathpalette\wthelper{#1}}
\newcommand*\wthelper[2]{%
        \hbox{\dimen@\accentfontxheight#1%
                \accentfontxheight#11.3\dimen@
                $\m@th#1\widetilde{#2}$%
                \accentfontxheight#1\dimen@
        }%
}

\newcommand*\accentfontxheight[1]{%
        \fontdimen5\ifx#1\displaystyle
                \textfont
        \else\ifx#1\textstyle
                \textfont
        \else\ifx#1\scriptstyle
                \scriptfont
        \else
                \scriptscriptfont
        \fi\fi\fi3
}
\makeatother

\begin{document}

\rhead{\thepage}
\lhead{\author}
\thispagestyle{empty}


\raggedbottom
\pagenumbering{arabic}
\setcounter{section}{0}


\title{Doubly pointed trisection diagrams and surgery on 2--knots}
\date{\today}

\author{David Gay}
\address{Euclid Lab\\ 160 Milledge Terrace\\ Athens, GA 30606
\newline
\indent Department of Mathematics\\ University
  of Georgia\\ Athens, GA 30605}
\email{d.gay@euclidlab.org}

\author{Jeffrey Meier}
\address{Department of Mathematics, University of Georgia, 
Athens, GA 30605}
\email{jeffrey.meier@uga.edu}
\urladdr{jeffreymeier.org} 
 
\begin{abstract}
	We study embedded spheres in 4--manifolds (2--knots) via doubly pointed trisection diagrams, showing that such descriptions are unique up to stabilization and handleslides, and we describe how to obtain trisection diagrams for certain cut-and-paste operations along 2--knots directly from doubly pointed trisection diagrams. The operations described are classical surgery, Gluck surgery, blowdown, and $(\pm4)$--rational blowdown, and we illustrate our techniques and results with many examples.
\end{abstract}

\maketitle

\section{Introduction}\label{sec:intro}

A \emph{$2$--knot} is a pair $(X,\Kk)$, where $X$ is a smooth, connected, compact, orientable 4--manifold, and $\Kk\subset X$ is a smoothly embedded $2$--sphere.  We sometimes refer to $\Kk$ as a $2$--knot, when the ambient manifold $X$ is understood from context. In this paper we use doubly pointed trisection diagrams (trisection diagrams augmented with pairs of points) to describe $2$--knots and to produce trisection diagrams for the $4$--manifolds resulting from various cut-and-paste operations on $2$--knots. The precise definitions needed to understand the main theorems will be given as the background material is developed in later sections.

In~\cite{MeiZup_Bridge-trisections_17} and~\cite{MeiZup_Bridge-trisections_}, the second author and Zupan showed that trisections of $4$--manifolds are the right setting in which to generalize bridge splittings of classical knots in dimension $3$, introducing the notion of a bridge trisection of an embedded surface in a $4$--manifold. Classical bridge splittings lead to multipointed Heegaard diagrams for knots, with $1$--bridge splittings giving doubly pointed diagrams, and this generalizes to dimension $4$. The following is a restatement of Theorem~1.2 and Corollary~1.3 of~\cite{MeiZup_Bridge-trisections_} as applied to $2$--knots.

\begin{theorem}[\cite{MeiZup_Bridge-trisections_}]
	Every 2--knot $(X,\Kk)$ admits a $1$--bridge trisection and, thus, can be described by a doubly pointed trisection diagram.
\end{theorem}

In this paper, we extend this existence statement to a uniqueness statement.

\begin{thmx}\label{thm:uniqueness}
	Any two $1$--bridge trisections of a given $2$--knot have a common stabilization.
\end{thmx}

This theorem has an immediate diagrammatic corollary.

\begin{corox}\label{coro:diag_uniqueness}
	Any two doubly pointed trisection diagrams for a given $2$--knot become slide-diffeo{\hskip0pt}omorphic after stabilization.
\end{corox}

In dimension three, a doubly pointed Heegaard diagram can be enriched with two arcs $\frak a$ and $\frak b$ connecting the points, one in the complement of the $\alpha$ curves and one in the complement of the $\beta$ curves; replacing the $S^0 \times B^2$ neighborhood of the points with a cylinder $B^1 \times S^1$ and extending the two arcs appropriately across this cylinder gives a Heegaard diagram for the result of integer surgery on the knot. The following can be seen as a $4$--dimensional generalization, involving an {\em arced trisection diagram} coming from a doubly pointed diagram; this will be defined carefully later but, for now, should be understood to be a trisection diagram on a surface with boundary (the trisection surface minus neighborhoods of the two points) augmented with three arcs $\frak a$ (red), $\frak b$ (blue) and $\frak c$ (green), in the complements, respectively, of the $\alpha$ (red), $\beta$ (blue) and $\gamma$ (green) curves. We consider here four cut-and-paste operations: \emph{sphere surgery}, in which $S^2 \times D^2$ is replaced by $B^3 \times S^1$, \emph{Gluck surgery}, in which $S^2 \times D^2$ is removed and glued back via the Gluck twist, \emph{$(\pm1)$--blowdown}, in which a neighborhood of a sphere of square $\pm1$ is replaced with $B^4$, and the \emph{$(\pm4)$--rational blowdown}, in which a neighborhood of a sphere of square $\pm4$ is replaced with a rational homology ball, as in~\cite{FinSte_Rational-blowdowns_97}.

\begin{thmx}\label{thm:surgery}
	Let $\DD$ be a doubly pointed $(g;k_1,k_2,k_3)$--trisection diagram for a 2--knot $(X,\Kk)$, and let $\DD^\circ$ be an associated arced trisection diagram for the knot exterior $E_\Kk$.  Consider the diagrams in the top row of Figure~\ref{fig:thm_gluings}.
	\begin{enumerate}
		\item If $\Kk\cdot\Kk = 0$, then the result $X(\Kk)$ of sphere surgery along $\Kk$ in $X$ is described by the $(g+1;k_1+1,k_2+1,k_3+1)$--trisection diagram $\DD^\circ\cup\DD_{B^3\times S^1}$.
		\item If $\Kk\cdot\Kk = 0$, then the result $X_*(\Kk)$ of Gluck surgery along $\Kk$ in $X$ is described by the $(g+1;k_1+1,k_2+1,k_3+1)$--trisection diagram $\DD^\circ\cup\DD_\frak a$.
		\item If $\Kk\cdot\Kk = 1$, then the result $X_{+1}(\Kk)$ of a $(+1)$--blowdown along $\Kk$ in $X$ is described by the $(g+1;k_1,k_2+1,k_3)$--trisection diagram $\DD^\circ\cup\DD_\frak a$.  Similarly, if $\Kk\cdot\Kk = -1$, then $X_{-1}(\Kk)$ is described by $\DD^\circ\cup\overline\DD_\frak a$.
		\item If $\Kk\cdot\Kk = 4$, then the result $X_{+4}(\Kk)$ of a 4--rational blowdown along $\Kk$ in $X$ is described by the $(g+2;k_1+1,k_2+1,k_3+1)$--trisection diagram $\DD^\circ\cup\DD_{B_{-4}}$.  Similarly, if $\Kk\cdot\Kk = -4$, then $X_{-4}(\Kk)$ is described by $\DD^\circ\cup\overline\DD_{B_{4}}$.
	\end{enumerate}
	These gluings of diagrams are illustrated in Figure~\ref{fig:thm_gluings}.
\end{thmx}

\begin{figure}[h!]
	\centering
	\includegraphics[width=\textwidth]{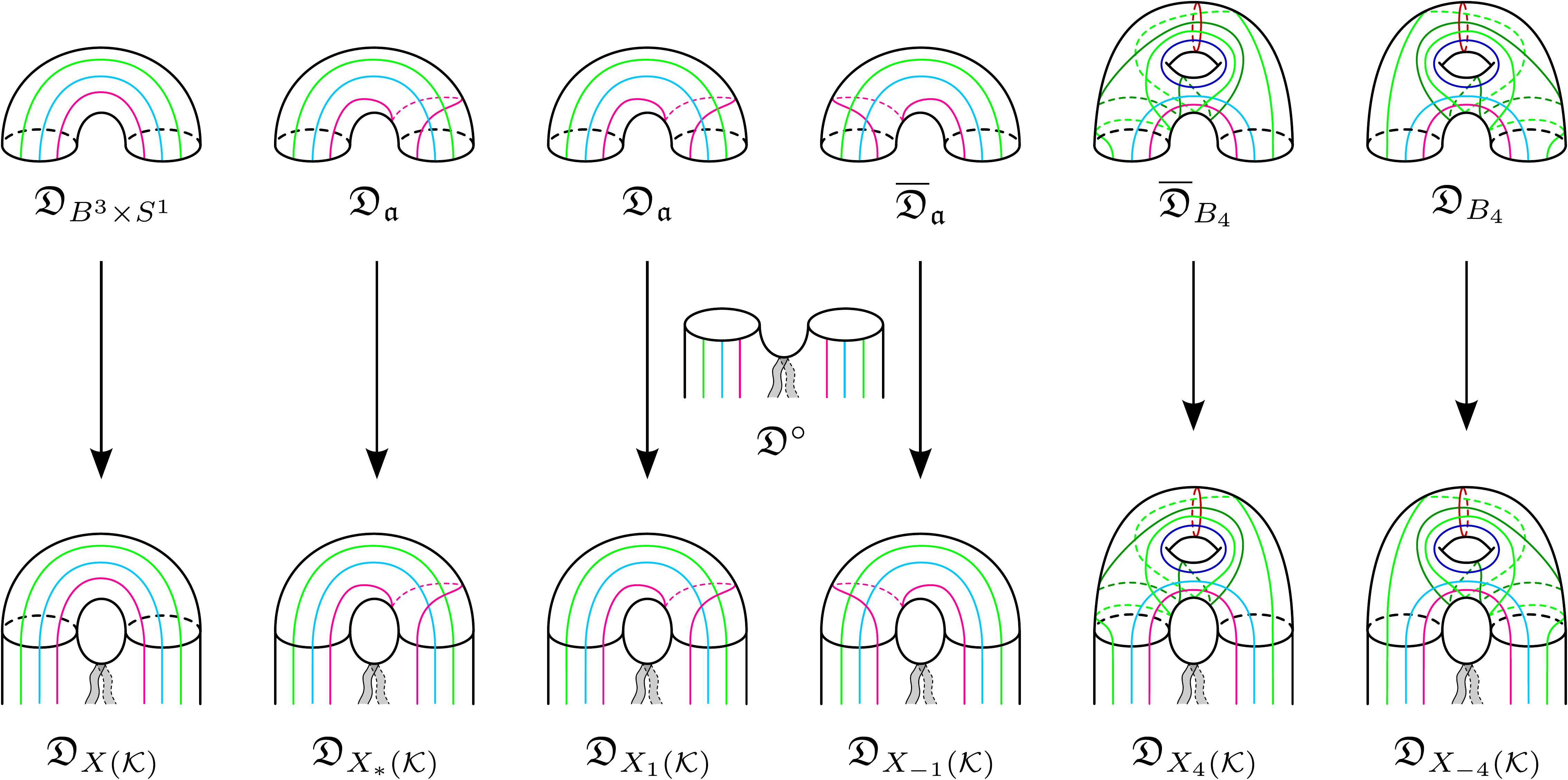}
	\caption{The diagrammatic gluings corresponding to the six (signed) surgery operations appearing in Theorem~\ref{thm:surgery}.  From left to right: sphere surgery, Gluck surgery, $(+1)$--blowdown, $(-1)$--blowdown, $(+4)$--rational blowdown, and $(-4)$--rational blowdown. The region crossing the saddle in $\DD^\circ$ is shaded to represent the fact that it may contain curves and arcs of any color.}
	\label{fig:thm_gluings}
\end{figure}

\begin{thmx}\label{thm:dual}
	Let $\DD^\circ$ be a 0--annular arced diagram for a 4--manifold $E$ with $\partial E\cong S^2\times S^1$.  Then, the diagrams $\DD$ and $\DD'$, as shown in Figure~\ref{fig:twin_diags}, are doubly pointed trisection diagrams for the only 2--knots $(X,\Kk)$ and $(X',\Kk')$ with $E_\Kk\cong E_{\Kk'}\cong E$.
\end{thmx}

\begin{figure}[h!]
	\centering
	\includegraphics[width=.5\textwidth]{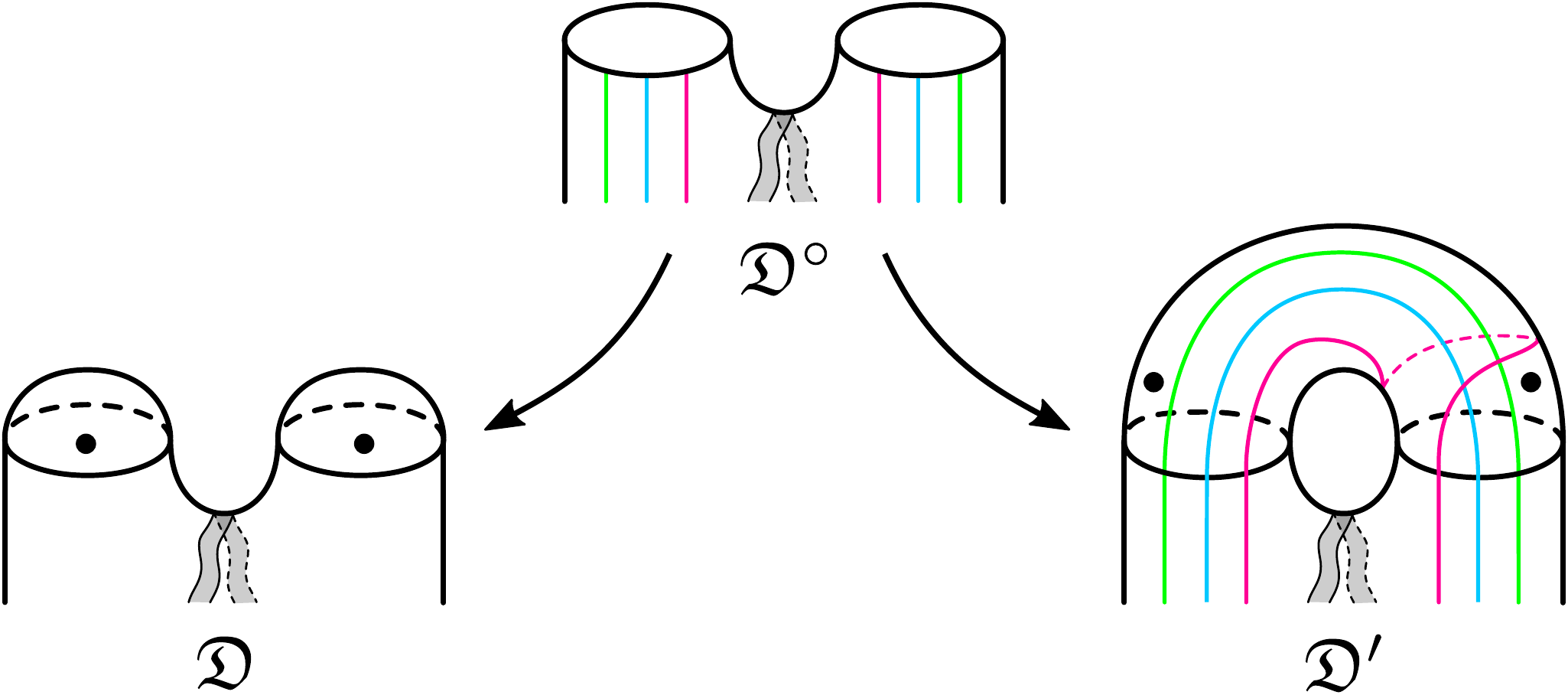}
	\caption{How to obtain doubly pointed trisection diagrams $\DD$ and $\DD'$ for the two 2--knots sharing a common exterior, given a $0$--annular arced trisection diagram $\DD^\circ$ for the exterior.}
	\label{fig:twin_diags}
\end{figure}

\subsection*{Organization}

The paper is organized as follows. In Section~\ref{sec:trisections}, we recall the foundations of the theory of trisections and trisection diagrams for both closed $4$--manifolds and $4$--manifolds with boundary, stating basic results about existence, uniqueness, and boundary data. In Section~\ref{sec:gluing}, we discuss the basic gluing results for trisections of $4$--manifolds with boundary and their diagrams, paying particular attention to the boundary parameterizations associated with the gluings. We also show how things work out particularly well when the boundary is a lens space with an annular open book. In Section~\ref{sec:bridge}, we give a detailed account of the adaptation of the theory of trisections to the setting of $2$--knots in $4$--manifolds and doubly pointed trisection diagrams and show how to get relative trisection diagrams for $2$--knot exteriors from doubly pointed diagrams (with lens space boundaries with annular open books).  Here, we prove Theorem~\ref{thm:uniqueness} and Corollary~\ref{coro:diag_uniqueness} and give a number of foundational examples that are needed for our surgery operations. In Section~\ref{sec:surgery}, we review the classical cut-and-paste operations involving 2--knots, and we prove Theorems~\ref{thm:surgery} and~\ref{thm:dual}.  We conclude in Section~\ref{sec:examples} by giving examples illustrating our diagrammatic techniques and results.

\subsection*{Acknowledgements}

The first author was supported by grant \#359873 from the Simons Foundation and NSF grant DMS-1664567. The second author was supported by NSF grants DMS-1400543 and DMS-1758087.  The second author would like to acknowledge Peter Lambert-Cole, Paul Melvin, Juanita Pinz\'on Caicedo, and Laura Starkston for their collaboration in a working group at the workshop \emph{Trisections and Low-Dimensional Topology} hosted by the American Institute of Mathematics in 2017, where bridge trisections of knotted surfaces in $\CP^2$ were first investigated.  We are grateful to AIM for this workshop, which has proved to be the origin of many interesting ideas.


\section{Trisections and their diagrams}\label{sec:trisections}

In this section, we review trisections of 4--manifolds, starting with closed 4--manifolds, then moving to the relative settings of compact 4--manifold with non-empty boundary.  Then, we discuss how these objects can be represented diagrammatically.

Given integers $g \geq k \geq 0$, consider the following standard manifolds:
\begin{itemize}
 \item $\Sigma_g = \#^g (S^1 \times S^1)$ is the standard closed, oriented genus $g$ surface;
 \item $H_g = \natural^g (S^1 \times B^2)$ is the standard, oriented genus $g$ handlebody, with $\partial H_g = \Sigma_g$;
 \item $Y_k = \#^k (S^2\times S^1)$ is the double of a genus $k$ handlebody; and
 \item $Z_k = \natural^k (B^3\times S^1)$ is the standard $4$--dimensional ``genus $k$'' $1$--handlebody, with $\partial Z_k = Y_k$.
\end{itemize}

The standard objects comprise the building blocks of a trisected 4--manifold.

\begin{definition}
	A \emph{$(g;k_1,k_2,k_3)$--trisection} of a $4$--manifold $X$ is a decomposition $X=X_1 \cup X_2 \cup X_3$, where, for each $i\in\Z_3$,
	\begin{enumerate}
	 	\item $X_i\cong Z_{k_i}$,
	 	\item $X_i\cap X_{i+1}\cong H_g$, and
	 	\item $X_1\cap X_2\cap X_3\cong \Sigma_g$.
	\end{enumerate}
	When $k_1=k_2=k_3=k$, this a {\em balanced} $(g,k)$--trisection, otherwise it is an {\em unbalanced} trisection. The {\em genus} of the trisection is $g$. The union of the three handlebodies
	$$(X_1 \cap X_2) \cup (X_2 \cap X_3) \cup (X_3 \cap X_1)$$
	is the {\em spine} of the trisection.  The handlebodies $X_1 \cap X_2$, $X_2 \cap X_3$, and $X_3 \cap X_1$ are denoted by $H_\beta$, $H_\gamma$, and $H_\alpha$, respectively, and called the \emph{spokes} of the trisection.  The common intersection $H_\alpha \cap H_\beta \cap H_\gamma$ is denoted by $\Sigma$ and called the \emph{core} of the trisection.
\end{definition}

See Figure~\ref{fig:rel_schematic} below for a schematic picture, with the caveat that this schematic more accurately depicts a trisection of a 4--manifold with boundary, which we will discuss in turn. Note that, for each $i$,
$$(X_{i-1} \cap X_i)\cup_\Sigma (X_i \cap X_{i+1})$$
is a genus $g$ Heegaard splitting of $\partial X_i \cong Y_k$, with $\Sigma$ oriented as $\partial(X_{i-1}\cap X_i)$.

\begin{definition}
\label{def:stab}
 Given a trisection $\Tt$ of a $4$--manifold $X$, for each $i\in\Z_3$, the {\em $i$--stabilization} of $\Tt$ is the trisection $\Tt'$ whose pieces $X_1'$, $X_2'$, and $X_3'$ are given in terms of the pieces $X_1$, $X_2$, and $X_3$ of $\Tt$ as follows. Choose a properly embedded boundary parallel arc $\omega$ in $X_{i-1} \cap X_{i+1}$. Let 
 	\begin{enumerate}
 		\item $X'_i = X_i \cup \overline{\nu(\omega)}$, while
 		\item $X'_{i-1} = X_{i-1} \setminus \nu(\omega)$ and
 		\item $X'_{i+1} = X_{i+1} \setminus \nu(\omega)$.
 	\end{enumerate}
  Note that this is well-defined up to isotopy independent of the choice of $\omega$, that it increases $g$ and $k_i$ by $1$ and does not affect $k_{i \pm 1}$.
\end{definition}

By stabilizing, any trisection can be made balanced. Noting that stabilization occurs in a ball, one could also describe stabilization as the connected sum with one of three standard genus one trisections of $S^4$; the only disadvantage of this is that it is not obviously an ambient operation inside a given $4$--manifold.

\begin{theorem}[\cite{GayKir_Trisecting-4-manifolds_16}]
 Every closed, connected, oriented $4$--manifold has a trisection, and any two trisections of the same $4$--manifold become isotopic after sufficiently many stabilizations.
\end{theorem}

\subsection{Trisections of compact 4--manifolds with non-empty boundary}\label{subsec:relative}\ 

Next, we recall the extension of the theory of trisections to compact 4--manifolds with connected,  non-empty boundary.  First, we need a more subtle understanding of the manifold $Y_l = \#^l(S^2\times S^1)$.  For $p\geq 0$ and $b\geq 1$, let $\Sigma_{p,b}$ denote the compact surface obtained by removing the interiors of $b$ small disks from the closed surface of genus $p$.  Consider the abstract open book $(\Sigma_{p,b},\id)$, whose page is $\Sigma_{p,b}$ and whose monodromy is the identity.  The total space of this abstract open book is $M_1\cong \#^{2p+b-1}(S^2\times S^1)$.  Pick three pages in this open book:  $P^+$, $P^-$, and $P^0$.  We augment $M_1$ by connected summing with a manifold $M_2$, which we take to be a copy of $\#^k(S^2\times S^1)$.  We assume $M_2$ is equipped with the standard genus $n$ Heegaard splitting.  We perform the connected sum at a point contained in the interior of $P^0$ in $M_1$ and at a point on the Heegaard surface $F$ in $M_2$.

We write 
$$Y_l = \#^l(S^2\times S^1) = M_1\# M_2 = \left(\#^{2p+b-1}(S^2\times S^1)\right)\#\left(\#^k(S^2\times S^1)\right),$$
and we note that $l = k + 2p+b-1$.  Moreover, we have the following decomposition:
$$Y_l = Y_{g,l;p,b}^-\cup Y_{g,l;p,b}^0\cup Y_{g,l;p,b}^+,$$
with pieces defined as follows.  We let $Y_{g,l;p,b}^0$ denote the compact portion of $M_1$ co-bounded by $P^-$ and $P^+$ and not containing $P^0$.  We let $Y_{g,l;p,b}^\pm$ denote the compact portion of $M_1\#M_2$ co-bounded by $P^\pm$ and $P_0\#F$.  The piece $Y_{g,l;p,b}^0$ is diffeomorphic to
$$\Sigma_{p,b}\times I/\!\sim,$$
where $(x,t)\sim(x,t')$ for all $x\in\partial\Sigma_{p,b}$ and $t,t'\in I$.  In other words, the vertical boundary has been collapsed, so $Y_{g,l;p,b}^0$ is a sort of \emph{lensed} product cobordism between the pages $P^-$ and $P^+$.  Each of $Y_{g,l;p,b}^\pm$ is a sort of lensed compression body, diffeomorphic to the result of first attaching $g-p$ 3--dimensional 2--handles to $\Sigma_{g,b}\times I$ along $\Sigma_{g,b}\times\{1\}$, then collapsing the vertical portion $(\partial\Sigma_{g,b})\times I$, as before.

\begin{definition}
	A \emph{relative $(g;k_1,k_2,k_3;p,b)$--trisection} of a 4--manifold $X$ is a decomposition
	$$X = X_1\cup X_2\cup X_3$$
	such that, for each $i\in\Z_3$, there is a diffeomorphism $\phi_i\colon X_i\to Z_{l_i}$ satisfying
	$$\phi_i(X_i\cap X_{i+1}) = Y_{g,l_i;p,b}^-, \hspace{.25in} \phi_i(X_i\cap\partial X) = Y_{g,l_i;p,b}^0,\hspace{.25in} \text{and}\hspace{.25in} \phi_i(X_i\cap X_{i-1}) = Y_{g,l_i;p,b}^+,$$
	where $l_i = k_i+2p+b-1$.
	We adopt the notation $H_\alpha$, $H_\beta$, $H_\gamma$, and $\Sigma$, just as in the closed case, and the concepts of \emph{(un)balanced}, \emph{genus}, and \emph{spine} are defined in the same way, as well.  A schematic of a relative trisection is shown in Figure~\ref{fig:rel_schematic}.
\end{definition}

A relative trisection $\Tt$ of a compact 4--manifold $X$ cuts $\partial X$ into three pieces, namely, the three pre-images of the $Y_{g,l_i;p,b}^0$.  Since each of these is a lensed product cobordism, we see that $\partial X$ inherits an open book decomposition from $\Tt$ with pages diffeomorphic to $\Sigma_{p,b}$ (three of which are given by the pre-images of the pages $P^\pm$ in the $Y_{k_i}$) and binding given by $\partial \Sigma\subset \partial X$.

\begin{figure}[h!]
	\centering
	\includegraphics[width=.35\textwidth]{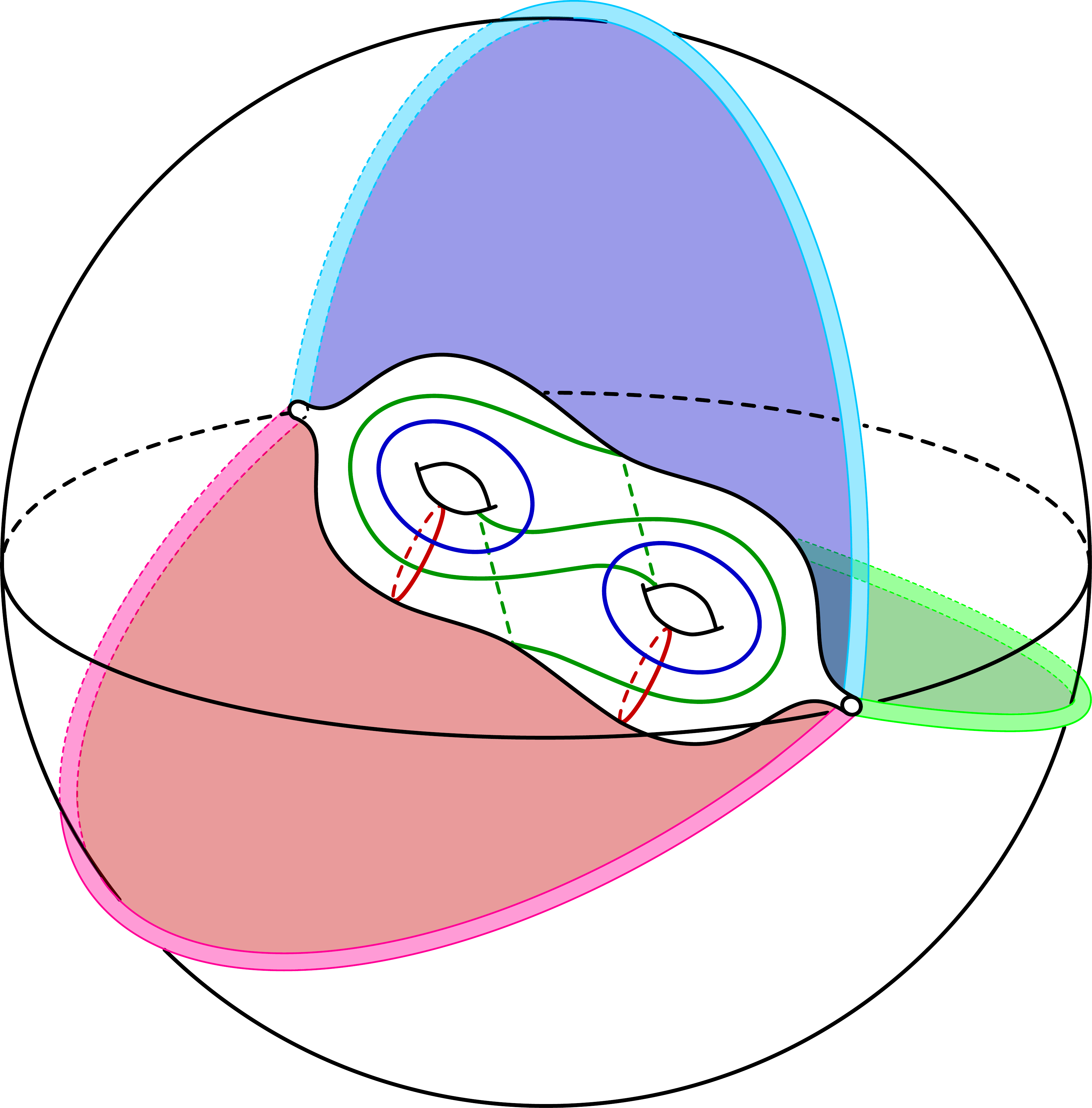}
	\caption{A schematic illustration of a relative trisection of a 4--manifold $X$, which is depicted as the black 3--ball.  Inside, the genus two trisection surface sits as a properly-embedded surface, intersecting the boundary of the 4--manifold in the two-component binding.  Shown on the boundary are three annular pages filling the binding, each of which is connected back to the trisection surface by a compression body.}
	\label{fig:rel_schematic}
\end{figure}

\begin{remark}\label{rmk:inside_k}
	The 4--dimensional 1--handlebodies $X_i$ can be thought of as being built (from the outside in) by starting with $P\times I\times I$ and attaching $k_i$ 1--handles.  Here, $P\times I$ is an interval product of a page in $\partial X\cap X_i$, while the second product with $I$ is a thickening into the interior of $X_i$. (Both thickenings are assumed to be lensed.)  It is thus that the piece $X_i$ has two relevant complexity measures:  $l_i$ is the measure of its total topological complexity, while $k_i$ measures the \emph{interior} complexity.  For example, a manifold with a trisection having $k_i=0$ can be built (relative to its boundary) without 1--handles.  Note that our convention of labeling the interior complexities by $k_i$ and the total complexity by $l_i$ is opposite that of~\cite{CasGayPin_Diagrams-for-relative_18}.
\end{remark}

The stabilization operations introduced above give well-defined stabilization operations for relative trisections, as well, and we have a similar existence and uniqueness statement.

\begin{theorem}[\cite{GayKir_Trisecting-4-manifolds_16}] \label{thm:relexistunique}
	Given a compact, connected, oriented $4$--manifold $X$ with connected, nonempty boundary, and an open book decomposition on $\partial X$, there is a trisection of $X$ inducing the given open book.
	Any two relative trisections for a 4--manifold $X$ that induce isotopic open book decompositions of $\partial X$ become isotopic after sufficiently many stabilizations.
\end{theorem}

\subsection{(Relative) trisection diagrams}\label{subsec:Tdiags}\ 

A key feature of the theory of trisections is that the data of a (relative) trisection can expressed diagrammatically via curves on surfaces. Here, we summarize this feature following~\cite{CasGayPin_Diagrams-for-relative_18}. We begin with the closed case.

\begin{definition}
\label{def:cut}
 A {\em cut system} on a closed, connected genus $g$ surface $\Sigma$ is a collection of $g$ disjoint simple closed curves on $\Sigma$ which collectively cut $\Sigma$ into a genus $0$ surface. Two cut systems are {\em slide-equivalent} if they are related by a sequence of handleslides. Two tuples $(\Sigma,\delta_1, \ldots, \delta_n)$ and $(\Sigma',\delta'_1, \ldots, \delta'_n)$, where each $\delta_i$ and $\delta'_i$ is a cut system, are {\em slide-diffeomorphic} if there is a diffeomorphism $\phi: \Sigma \to \Sigma'$ such that each $\phi(\delta_i)$ is slide-equivalent to $\delta'_i$.
\end{definition}
As is well known, a cut system $\alpha$ on $\Sigma$ determines (up to diffeomorphism rel. boundary) a handlebody $H_\alpha$ with $\partial H_\alpha = \Sigma$, every handlebody $H$ with $\partial H = \Sigma$ is $H_\alpha$ for some cut system $\alpha$, and $H_\alpha$ and $H_{\alpha'}$ are diffeomorphic rel. boundary if and only if $\alpha$ and $\alpha'$ are slide-equivalent~\cite{Joh_Topology-and-combinatorics_95}.

\begin{definition}
\label{def:closed}
 A {\em Heegaard diagram} is a triple $(\Sigma,\alpha,\beta)$ where $\Sigma$ is a surface and each of $\alpha$ and $\beta$ are cut systems on $\Sigma$. The Heegaard diagram appearing as the top graphic of Figure~\ref{fig:Heeg_diags} is called the \emph{trivial $(g,k)$--diagram}.  (Ignore the double points in the center of this graphic for now.)  A {\em Heegaard triple} is a $4$--tuple $(\Sigma, \alpha,\beta,\gamma)$ where $\Sigma$ is a surface and each of $\alpha$, $\beta$ and $\gamma$ are cut systems on $\Sigma$. A $(g;k_1,k_2,k_3)$--{\em trisection diagram} is a genus $g$ Heegaard triple $(\Sigma, \alpha,\beta,\gamma)$ such that each of $(\Sigma,\alpha,\beta)$, $(\Sigma,\beta,\gamma)$ and $(\Sigma,\gamma,\alpha)$ is slide-diffeomorphic to the trivial $(g,k_i)$--diagram, where $k_1$, $k_2$, and $k_3$ count the number of parallel curves in the diagram of the figure corresponding to $(\Sigma,\alpha,\beta)$,  $(\Sigma,\beta,\gamma)$ and $(\Sigma,\gamma,\alpha)$, respectively. (Again, we ignore the central double points for now.)
\end{definition}

\begin{figure}[h!]
	\centering
	\includegraphics[width=.9\textwidth]{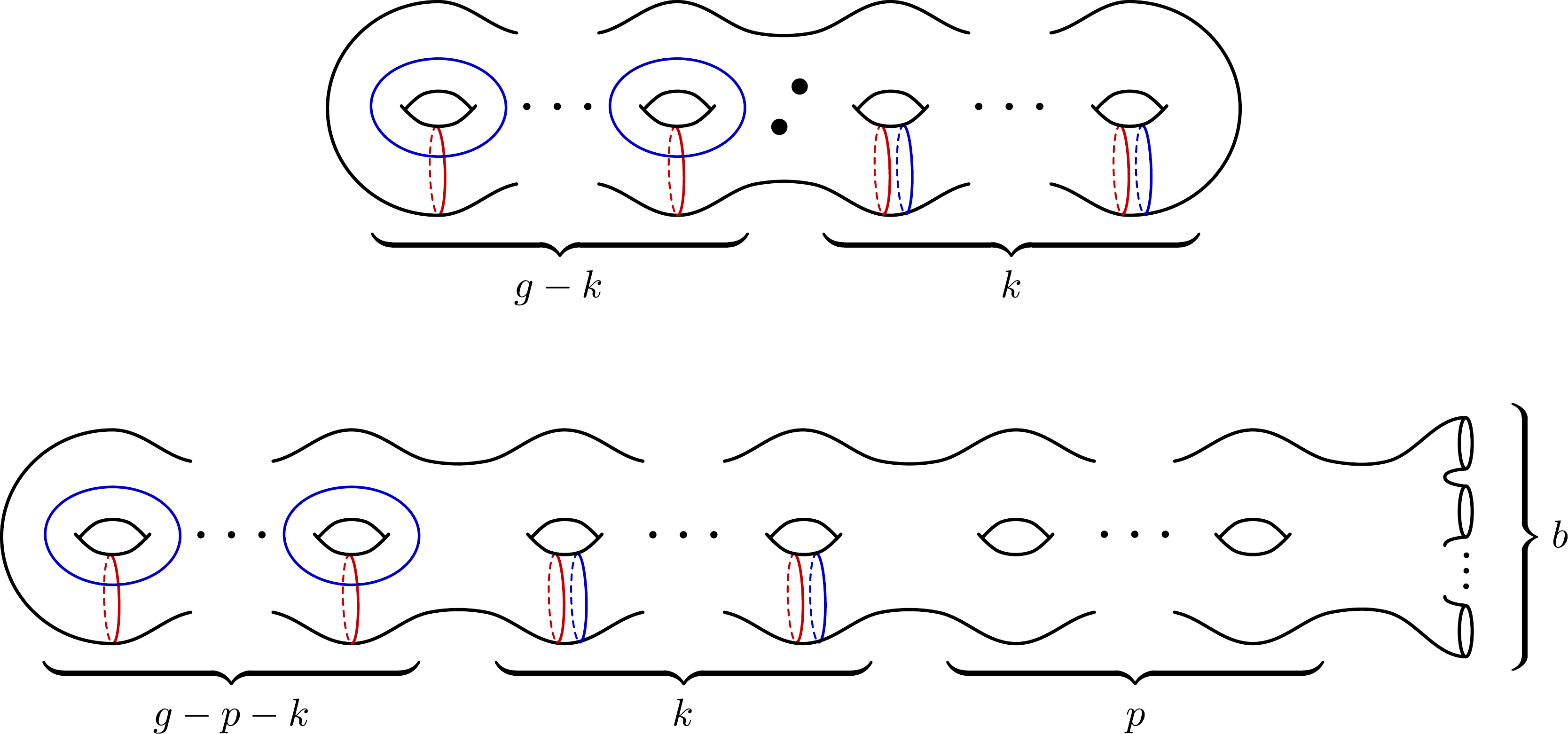}
	\caption{(Top) The trivial $(g,k)$--diagram; each of the three Heegaard diagrams comprising a genus $g$ trisection diagram is slide-diffeomorphic to this diagram for some value $k$.  (Bottom) The trivial $(g,k;p,b)$--diagram; each pair of genus $p$ cut systems for a genus $g$ relative trisection diagram is slide-diffeomorphic to this diagram for some value of $k$.}
	\label{fig:Heeg_diags}
\end{figure}

A Heegaard diagram $(\Sigma,\alpha,\beta)$ determines a closed, connected oriented $3$--manifold $H_\alpha \cup_\Sigma \overline H_\beta$, well-defined up to orientation preserving diffeomorphism by the slide-diffeomorphism type of the diagram, and any closed, connected, oriented $3$--manifold is described by a diagram.  The trivial $(g,k)$--diagram determines the genus $g$ Heegaard splitting of $\#^k(S^2\times S^1)$.

A Heegaard triple $(\Sigma,\alpha,\beta,\gamma)$ determines a compact, connected, oriented $4$--manifold with three boundary components by gluing $I \times H_\alpha$, $I \times H_\beta$ and $I \times H_\gamma$ to $D^2 \times \Sigma$ along $I_\alpha \times \Sigma$, $I_\beta \times \Sigma$ and $I_\gamma \times \Sigma$, where $I_\alpha$, $I_\beta$ and $I_\gamma$ are three disjoint arcs in $S^1$ (with the $\alpha,\beta,\gamma$ order being {\em clockwise} around $S^1$). When this Heegaard triple is a trisection diagram $\DD$, each of the three boundary components is a connected sum of copies of $S^2\times S^1$, which can be filled in uniquely with boundary connected sums of copies of $B^3\times S^1$~\cite{LauPoe_A-note-on-4-dimensional_72}, and this closed $4$--manifold is denoted $X(\DD)$. Note that $X(\DD)$ comes with an implicit $(g;k_1,k_2,k_3)$ trisection $X(\DD)=X_1 \cup X_2 \cup X_3$, such that $\Sigma = X_1 \cap X_2 \cap X_3 \subset X(\DD)$ is the trisection surface and such that $\alpha$, $\beta$ and $\gamma$, resp., bound disks in the handlebodies $X_3 \cap X_1$, $X_1 \cap X_2$ and $X_2 \cap X_3$, resp.

The diagrammatic content of~\cite{GayKir_Trisecting-4-manifolds_16} is that every closed, connected oriented $4$--manifold is $X(\DD)$ for some trisection diagram $\DD$, that slide-diffeomorphic diagrams give diffeomorphic $4$--manifolds, and that two diagrams give diffeomorphic $4$--manifolds if and only if they are related by slide-diffeomorphism and {\em stabilization}, where diagrammatic stabilization is connected summing with one of the three genus one trisection diagrams for $S^4$ shown in Figure~\ref{fig:BasicStab}. Furthermore, any given trisection of the original $4$--manifold is actually the trisection coming from a diagram.

\begin{figure}[h!]
	\centering
	\includegraphics[width=.65\textwidth]{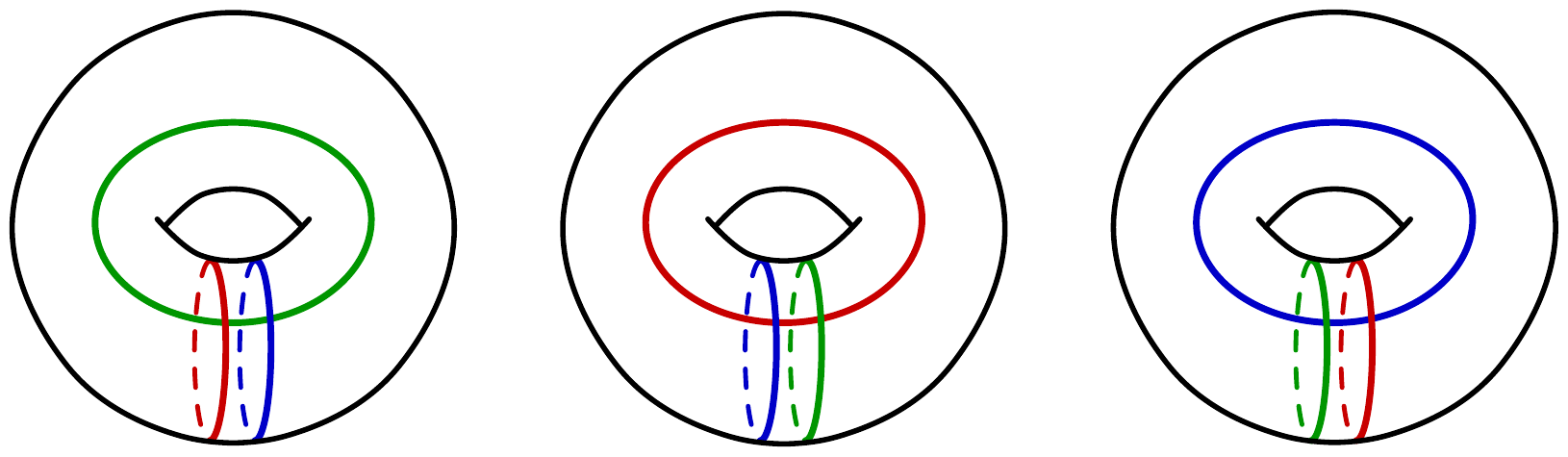}
	\caption{The three genus one trisection diagrams for $S^4$, each of which is unbalanced.  Any given trisection diagram or relative trisection diagram can be stabilized by forming the connected sum with one of these three diagrams.}
	\label{fig:BasicStab}
\end{figure}

\begin{definition}
 A {\em genus $p$ cut system} on a compact, connected, orientable genus $g$ surface $\Sigma$ with $b$ boundary components is a collection of $g-p$ disjoint simple closed curves on $\Sigma$ which collectively cut $\Sigma$ into a connected genus $p$ surface. The notions of \emph{slide-equivalent} and \emph{slide-diffeomorphic} carry over verbatim from Definition~\ref{def:cut}.
\end{definition}

In this more general setting, a genus $p$ cut system on a genus $g$ surface $\Sigma$ with $b$ boundary components determines (up to diffeomorphism rel. boundary) a compression body $C_\alpha$ with $\partial C_\alpha = \Sigma \cup (I \times \partial \Sigma) \cup \Sigma_\alpha$, where $\Sigma_\alpha$ is the result of surgering $\Sigma$ along $\alpha$. Every such compression body is $C_\alpha$ for some cut system $\alpha$, and $C_\alpha$ and $C_\alpha'$ are diffeomorphic rel. boundary if and only if $\alpha$ and $\alpha'$ are slide-equivalent. (Here, ``rel. boundary'' really means relative to the boundary respecting the decomposition of the boundary as $\Sigma \cup (I \times \partial \Sigma) \cup \Sigma_\alpha$, i.e. respecting the structure of $C_\alpha$ as a relative cobordism from $\Sigma$ to $\Sigma_\alpha$.)

\begin{definition}
 A $(g;k_1,k_2,k_3;p,b)$--{\em relative trisection diagram} is a $4$--tuple $(\Sigma, \alpha,\beta,\gamma)$ where $\Sigma$ is a genus $g$ compact, connected surface with $b$ boundary components, $\alpha$, $\beta$ and $\gamma$ are genus $p$ cut systems on $\Sigma$, and each of $(\Sigma,\alpha,\beta)$, $(\Sigma,\beta,\gamma)$ and $(\Sigma,\gamma,\alpha)$ is slide-diffeomorphic to the trivial $(g,k_i;p,b)$--diagram shown in Figure~\ref{fig:Heeg_diags}.
\end{definition}

The first author, with Castro and Pinzon-Caicedo in~\cite{CasGayPin_Diagrams-for-relative_18}, showed that relative trisection diagrams uniquely determine relatively trisected $4$--manifolds with boundary. In other words, for every $(g;k_1,k_2,k_3;p,b)$ relative trisection diagram $\DD = (\Sigma, \alpha,\beta,\gamma)$, there is a $(g;k_1,k_2,k_3;p,b)$--trisected $4$--manifold with boundary, $X(\DD)=X_1 \cup X_2 \cup X_3$, such that $\Sigma = X_1 \cap X_2 \cap X_3$ and such that $\alpha$, $\beta$ and $\gamma$ bound disks in the compression bodies $X_1 \cap X_2$, $X_2 \cap X_3$ and $X_3 \cap X_1$, respectively; and $X(\DD)$ is uniquely determined up to trisected diffeomorphism. In particular this also determines a $3$--manifold $\partial X(\DD)$ equipped with a genus $p$ open book decomposition with $b$ binding components, well-defined up to open book preserving diffeomorphism. Furthermore, the existence and uniqueness result in the relative case in~\cite{GayKir_Trisecting-4-manifolds_16} translates diagrammatically into the statement that for every $4$--manifold $X$ equipped with an open book decomposition of $\partial X$, there is a relative trisection diagram $\DD$ with $X(\DD) \cong X$ that induces the given open book on the boundary and that any two diagrams giving diffeomorphic $4$--manifolds with diffeomorphic boundary open books become slide-diffeomorphic after diagrammatic stabilization. The stabilization in the relative case is exactly the same as the closed case, being an interior connected sum with a diagram in Figure~\ref{fig:BasicStab}.

What is not immediately clear from the above is how to understand the $3$--manifold and its open book decomposition in terms of a diagram, and for this we need to add arcs to our cut systems.
\begin{definition}
 Given a genus $p$ cut system $\alpha$ on $\Sigma$, an {\em arc system relative to} $\alpha$ is a collection $\frak a$ of $2p$ properly embedded arcs in $\Sigma$, disjoint from $\alpha$, such that cutting along $\frak a$ and surgering along $\alpha$ turns $\Sigma$ into a disk. If $\frak a$ and $\frak a'$ are  arc systems relative to cut systems $\alpha$ and $\alpha'$, respectively, we say that $(\alpha,\frak a)$ is slide-equivalent to $(\alpha',\frak a')$ if the one can be transformed to the other by ordinary handleslides on the cut systems and by sliding arcs from the arc system over curves from the cut system. Note that we {\em do not allow} the sliding of arcs over arcs, nor isotopies that move points on $\partial \Sigma$. 
\end{definition}

\begin{definition}
	An {\em arced relative trisection diagram} (or \emph{arced trisection diagram}, or even \emph{arced diagram}, for short) is a tuple $(\Sigma,\alpha,\beta,\gamma,\frak a,\frak b,\frak c)$ such that $(\Sigma,\alpha,\beta,\gamma)$ is a relative trisection diagram, $\frak a$ (resp. $\frak b$, resp. $\frak c$) is an arc system relative to $\alpha$ (resp. $\beta$, resp. $\gamma$) and such that we have the following pairwise standardness conditions:
	\begin{enumerate}
		\item $(\Sigma,\alpha,\beta,\frak a,\frak b)$ is slide-equivalent to some $(\Sigma,\alpha',\beta',\frak a',\frak b')$ such that $(\Sigma,\alpha',\beta')$ is diffeomorphic to the trivial $(g,k_i;p,b)$--diagram and $\frak a' = \frak b'$.
		\item $(\Sigma,\beta,\gamma, \frak b,\frak c)$ is slide-equivalent to some $(\Sigma,\beta',\gamma',\frak b',\frak c')$ such that $(\Sigma,\beta',\gamma')$ is diffeomorphic to the trivial $(g,k_i;p,b)$--diagram and $\frak b' = \frak c'$.
	\end{enumerate}
	Observe that $\partial \frak a = \partial \frak b = \partial \frak c$.
\end{definition}

\begin{definition}
	A {\em completed arced relative trisection diagram} (or \emph{completed arced diagram} for short) is a tuple $(\Sigma,\alpha,\beta,\gamma,\frak a,\frak b,\frak c, \frak a^*)$ such that $(\Sigma,\alpha,\beta,\gamma,\frak a,\frak b,\frak c)$ is an arced diagram and such that $(\Sigma,\gamma, \alpha, \frak c,\frak a^*)$ is slide-equivalent to some $(\Sigma,\gamma',\alpha'',\frak c',\frak a'')$ such that $(\Sigma,\gamma',\alpha'')$ is diffeomorphic to the trivial $(g,k;p,b)$--diagram.
\end{definition}

Note that one cannot in general assume $\frak a = \frak a^*$ in a completed diagram, since, as will become clear shortly, this would imply trivial monodromy.

\begin{definition}
	An {\em abstract open book} is a pair $(P,\phi)$ where $P$ (the \emph{page}) is a compact, connected, oriented surface with nonempty boundary and $\phi$ (the \emph{monodromy}) is a self-diffeomorphism of $P$ which is the identity on $\partial P$. Two abstract open books with the same page are {\em isotopic} if their monodromies are isotopic relative to the boundary of the page. Two abstract open books $(P,\phi)$ and $(P',\phi')$ are diffeomorphism isotopic if there is a diffeomorphism between $P$ and $P'$ with respect to which the monodromies are isotopic. An abstract open book $(P,\phi)$ determines a model $3$--manifold with open book decomposition $M(P,\phi)=[0,1] \times P/\!\sim$ where $(p,1) \sim (\phi(p),0)$ for all $p \in P$ and $(p,s) \sim (p,t)$ for all $p \in \partial P$ and all $s,t \in [0,1]$.
\end{definition}

\begin{definition}
	Consider a relative trisection diagram $\DD = (\Sigma, \alpha,\beta,\gamma)$. The {\em page associated to $\DD$} is the surface $P_\DD = \Sigma_\alpha$ obtained by surgering all the $\alpha$ curves; this has genus $p$ with $b$ boundary components, if $\Sigma$ had $b$ boundary components and $\alpha$ is a genus $p$ cut system. Now consider a completed arced diagram $\DD^* = (\Sigma,\alpha,\beta,\gamma,\frak a,\frak b,\frak c, \frak a^*)$ whose underlying relative diagram is $\DD$. Note that both $\frak a$ and $\frak a^*$ descend to well-defined (up to isotopy rel. boundary) arc systems on $P_\DD$  The {\em monodromy associated to $\DD^*$} is the diffeomorphism (well-defined up to isotopy rel. boundary) $\phi(\DD^*): P_\DD \to P_\DD$ taking $\frak a$ to $\frak a^*$ and fixing $\partial P_\DD$ pointwise. Thus, the {\em abstract open book associated to $\DD^*$} is the pair $(P_\DD,\phi_{\DD^*})$.
\end{definition}

\begin{theorem}[Castro-Gay-Pinzon~\cite{CasGayPin_Diagrams-for-relative_18}]
 For any relative diagram $\DD = (\Sigma;\alpha,\beta,\gamma)$ and any arc system $\frak a$ relative to $\alpha$ on $\Sigma$, there exist cut systems $\frak b$, $\frak c$ and $\frak a^*$ such that $\DD^* = (\Sigma;\alpha,\beta,\gamma;\frak a,\frak b,\frak c, \frak a^*)$ is a completed arced diagram. Furthermore, the abstract open book $(P_\DD,\phi_{\DD^*})$ is uniquely determined (up to isotopy) by the original relative diagram $\DD$ and in fact the model $3$--manifold $M(P_\DD,\phi_{\DD^*})$ is diffeomorphic, respecting open books, to $\partial X(\DD)$.
\end{theorem}

In light of the fact that $\phi_{\DD^*}$ does not depend on the choice of arcs used to convert $\DD$ to $\DD^*$, we will henceforth write $\phi_\DD$, without loss of specificity.

\section{Gluing trisections along open books}\label{sec:gluing}

In this section, we carefully set up the machinery needed to intelligibly glue relative trisections together when their boundaries are equipped with induced diffeomorphic open book decompositions. The bulk of this section essentially restates gluing results from Section~2 of~\cite{CasOzb_Trisections-of-4--manifolds_17}, but the conclusion of this section is to look more carefully at the special case when we are gluing along lens spaces with annular open books. First we expand on the above theorem slightly.

\begin{lemma}
 Given an arced relative trisection diagram $\DD = (\Sigma;\alpha,\beta,\gamma;\frak a,\frak b,\frak c)$, let $X(\DD)$ be the associated trisected $4$--manifold with boundary. Let $(P_\DD,\phi_\DD)$ be the associated abstract open book and let $M(P_\DD,\phi_\DD)$ be the associated model open book. Then there is a canonical (up to isotopy) diffeomorphism
$$\Psi(\DD) \colon \partial X(\DD) \to M(P_\DD,\phi_\DD).$$
In particular, $\Psi(\DD)$ is uniquely determined up to isotopy by the diagram $\DD$, and does not depend on the choice of arcs but only on the underlying relative diagram.
\end{lemma}

\begin{proof}
 This is already proved as Theorem~5 in~\cite{CasGayPin_Diagrams-for-relative_18}, only that the canonicity of the diffeomorphism is not emphasized there. Once one understands that such a diffeomorphism $\Psi(\DD)$ exists, observe simply that $P_\DD = \Sigma_\alpha$ naturally sits as a submanifold of $\partial X(\DD)$, as one end of the $\alpha$ compression body. We also identify $P_\DD$ naturally with $\{0\} \times P_\DD \subset M(P_\DD,\phi_\DD)$. Then we note that any self-diffeomorphism of a $3$--manifold which is the identity on a fixed page of a fixed open book decomposition is necessarily isotopic to the identity.
\end{proof}

\begin{definition}
	Two arced relative trisection diagrams
	$$\DD = (\Sigma;\alpha,\beta,\gamma;\frak a,\frak b,\frak c) \text{\ \ \ and \ \ \ }\DD' = (\Sigma';\alpha',\beta',\gamma';\frak a',\frak b',\frak c'),$$
	together with an orientation reversing diffeomorphism $f\colon\partial(\Sigma,\frak a)\to\partial(\Sigma',\frak a')$,
	are called \emph{gluing compatible} if there is an orientation reversing diffeomorphism
	$$\psi_f(\DD,\DD') \colon P_\DD = \Sigma_\alpha \to P_{\DD'} = \Sigma'_{\alpha'}$$
	that extends $f$, takes $\frak a$ to $\frak a'$, and commutes with the monodromy diffeomorphisms $\phi_\DD$ and $\phi_{\DD'}$. 
	Note that if such a diffeomorphism exists, it is uniquely determined up to isotopy because the arc systems $\frak a$ and $\frak a'$, respectively, cut $P_\DD$ and $P_{\DD'}$, respectively, into disks. We call the diffeomorphism $f$ a \emph{compatible gluing}.
\end{definition}

The compatible gluing $f$ will be implicit in our diagrams in what follows, so we will suppress it from the notation, writing $\psi(\DD,\DD')$ for $\psi_f(\DD,\DD')$. In light of this set-up, the following is Proposition~2.12 of~\cite{CasOzb_Trisections-of-4--manifolds_17}.

\begin{proposition}
\label{prop:glue}
	Given two gluing compatible arced diagrams $\DD$ and $\DD'$, consider the diffeomorphism
	$$\Psi(\DD,\DD')\colon M(P_\DD,\phi_\DD) \to M(P_{\DD'},\phi_{\DD''})$$
	defined by sending each $\{t\} \times P_\DD$ to $\{t\} \times P_{\DD'}$ via $\psi(\DD,\DD')$, and the associated diffeomorphism
	$$\Upsilon(\DD,\DD') \colon \partial X(\DD) \to \partial X(\DD')$$
	defined by $\Upsilon(\DD,\DD')= \Psi(\DD')^{-1} \circ \Psi(\DD,\DD') \circ \Psi(\DD)$. Also consider the closed Heegaard triple $\DD \cup \DD'$ obtained by gluing $\Sigma$ to $\Sigma'$ so as to match corresponding end points of arcs. Then this is a trisection diagram and the corresponding closed $4$--manifold is the $4$--manifold $X(\DD\cup\DD')$ built by gluing $X(\DD)$ to $X(\DD')$ using the orientation reversing diffeomorphism $\Upsilon(\DD,\DD')^{-1} = \Upsilon(\DD',\DD)$.
\end{proposition}

Note that, because our definition of arced diagram requires that the arcs $\frak a$, $\frak b$, and $\frak c$ all share common endpoints, the resulting closed diagram will have triple points of intersection among the curves of $\alpha$, $\beta$, and $\gamma$ at the gluing points. These can of course be perturbed away after the gluing. In particular, when the page of our open book is an annulus, we can consistently displace the end points of $\frak a$, $\frak b$, and $\frak c$ following a counterclockwise orientation along one boundary component and a clockwise orientation on the other, {\em before} the gluing, so that we immediately avoid the triple points. This is what we will do in later sections, and in all figures.

\begin{lemma}
\label{lem:composition}
 Given three gluing compatible diagrams $\DD$, $\DD'$ and $\DD''$, the three diffeomorphisms $\Upsilon(\DD,\DD')$, $\Upsilon(\DD',\DD'')$ and $\Upsilon(\DD,\DD'')$ satisfy the composition rule:
 $$\Upsilon(\DD,\DD'') = \Upsilon(\DD',\DD'') \circ \Upsilon(\DD,\DD').$$
\end{lemma}

\begin{proof}
 Simply follow the definitions of all the maps involved.
\end{proof}

\subsection{Relative trisections inducing annular open book decompositions}
\label{subsec:annular}

We now focus our attention on a very special case, when the pages of the open book decompositions being considered are annuli, in which case the $3$--manifolds involved are lens spaces $L(p,1)$.

\begin{definition}
 A(n arced) relative trisection diagram $\DD$ is {\em annular} if the page $P_\DD = \Sigma_\alpha$ is an annulus. We say that $\DD$ is {\em $p$--annular} if the $3$--manifold $M(P_\DD,\phi_\DD)$ is diffeomorphic to the lens space $L(p,1)$. Equivalently, $\DD$ is $p$--annular if the page is an annulus and the monodromy is $\tau_C^p$ where $\tau_C$ is a Dehn twist around the core $C$ of the annulus.
\end{definition}

\begin{lemma}
\label{lem:compatible}
 Given a fixed $p$, any two $p$--annular arced relative diagrams $\DD$ and $\DD'$ are gluing compatible, and the associated gluing map $\Upsilon(\DD,\DD')$ is well-defined (up to isotopy) and {\em independent} of the choice of arcs extending the relative diagrams to arced relative diagrams.
\end{lemma}

\begin{proof}
 The fact that $\DD$ and $\DD'$ are gluing compatible is because the model manifold $M(P,\phi)$ for an abstract open book $(P,\phi)$, where $P$ is an annulus, is diffeomorphic to $L(p,1)$ if and only if $\phi$ is $\tau_C^p$, and since the mapping class group of the annulus is $\mathbb{Z}$, generated by $\tau_C$, any orientation reversing diffeomorphism from $P_\DD$ to $P_{\DD'}$  commutes with monodromies $\phi_\DD$ and $\phi_{\DD'}$ (up to isotopy rel. boundary), and hence there is such a diffeomorphism $\psi$ taking $\frak a$ to $\frak a'$.
 
 Now suppose $\DD_1$ and $\DD_2$ are different arced diagrams extending the same underlying relative diagram, and similarly that $\DD'_1$ and $\DD'_2$ are different arced diagrams extending the same underlying relative diagram. Thus $P_{\DD_1} = P_{\DD_2}$ and $P_{\DD'_1}=P_{\DD'_2}$. The two gluing maps we wish to compare are
 $$\Upsilon(\DD_1,\DD'_1) = \Psi(\DD'_1) \circ \Psi(\DD_1,\DD'_1) \circ \Psi(\DD_1)$$
 and
 $$\Upsilon(\DD_2,\DD'_2) = \Psi(\DD'_2) \circ \Psi(\DD_2,\DD'_2) \circ \Psi(\DD_2).$$
 We will show that $\Upsilon(\DD_2,\DD'_2)^{-1} \circ \Upsilon(\DD_1,\DD'_1)$ is isotopic to the identity. Since we know that the boundary parametrizations $\Psi(\DD_1)$, $\Psi(\DD_2)$, $\Psi(\DD'_1)$, and $\Psi(\DD'_2)$ are well-defined up to isotopy only by the underlying relative diagrams, i.e. are independent of choice of arcs, we need only show that $\Psi(\DD_2,\DD'_2)^{-1} \circ \Psi(\DD_1,\DD'_1)$ is isotopic to the identity. But this is just the map from $M(P_{\DD_1},\phi_{\DD_1})$ to $M(P_{\DD_2},\phi_{\DD_2}) = M(P_{\DD_1},\phi_{\DD_1})$ defined by a map from $P_{\DD_1}=P_{\DD_2}$ to itself which maps arc $\frak a_1$ to $\frak a_2$. In other words, this maps pages to pages and is the identity on the binding. Thus, up to isotopy, this self-map of an open book is determined by its effect on one page, and on that one page it acts by some power of a boundary parallel Dehn twist. Such a self-map of an open book is isotopic to the identity via an isotopy which rotates the corresponding binding component as many times as needed to undo the Dehn twists.
\end{proof}

Note that we have in fact proved the independence of choice of arcs whenever the page has mapping class group generated by boundary Dehn twists, but this is true only in one case other than the annulus, namely the case of a pair of pants.

\section{Bridge trisections of 2--knots}\label{sec:bridge}

The concept of a trisection can be extended to the setting of knotted surfaces in 4--manifolds.  This was first carried out in~\cite{MeiZup_Bridge-trisections_17} for surfaces in $S^4$, before being extended to the general setting in~\cite{MeiZup_Bridge-trisections_}, where we refer the reader for general details.  In the present note, we will restrict our attention to a special type of generalized bridge trisection: namely, 1--bridge trisections of 2--knots. Recall that a \emph{2--knot} is a smoothly embedded 2--sphere $\Kk$ in a smooth, orientable, connected, closed 4--manifold $X$, which we will alternately denote by $(X,\Kk)$ or simply $\Kk$, when the ambient space is clear from context.

A smooth disk $D$ that is properly embedded in $Z_k = \natural^k(B^3\times S^1)$ is called \emph{trivial} if there is an isotopy of $D$ that fixes $\partial D$ point-wise and pushes $D$ into $\partial Z_k$.  It follows that $U = \partial D$ is an unknot in $Y_k = \partial Z_k = \#^k(S^2\times S^1)$.

\begin{definition}
	A \emph{1--bridge $(g;k_1,k_2,k_3)$--trisection} of a 2--knot $(X,\Kk)$ is a decomposition
	$$(X,\Kk) = (X_1,D_1)\cup (X_2,D_2)\cup (X_3,D_3),$$
	where
	\begin{enumerate}
		\item $X=X_1\cup X_2\cup X_3$ is a $(g;k_1,k_2,k_3)$--trisection,
		\item $D_i$ is a trivial disk in $X_i$, and
		\item $\Kk\cap H_i$ is a properly embedded arc.
	\end{enumerate}
\end{definition}

The following is the restriction of Theorem~1.2 of~\cite{MeiZup_Bridge-trisections_} to the setting of a 2--knot $\Kk$ in a closed 4--manifold $X$.

\begin{theorem}[\cite{MeiZup_Bridge-trisections_}]
	Every pair $(X,\Kk)$ admits a 1--bridge trisection.
\end{theorem}

The notion of $i$--stabilization discussed above yields a natural notion of stabilization for 1--bridge trisections:  Simply perform an $i$--stabilization of the underlying 4--manifold trisection away from the 2--knot $\Kk$.  Thus, we adopt the terminology of Definition~\ref{def:stab} with respect to a 1--bridge trisection $\Tt$ for $(X,\Kk)$.  In Subsection~\ref{subsec:exteriors}, we prove the following.

\begin{reptheorem}{thm:uniqueness}
	Any two 1--bridge trisections of a given 2--knot have a common stabilization.
\end{reptheorem}

It is not difficult to extend the notion of 1--bridge trisections to the setting of knotted spheres in compact 4--manifolds with boundary, equipped with relative trisections.  We will make use of this generalization in Subsection~\ref{subsec:core}.  However, we will generally be interested in 1--bridge trisections of knotted spheres in closed 4--manifolds -- i.e., 2--knots.

We now describe how 1--bridge trisections can be encoded diagrammatically.  First, we recall a familiar notion from 3--manifold topology.

\begin{definition}
	A \emph{doubly pointed Heegaard diagram} is a tuple $(\Sigma;\alpha,\beta, \bold x_+,\bold x_-)$ where $\Sigma$ is a closed surface, each of $\alpha$ and $\beta$ is a cut system for $\Sigma$, and $\bold x_+$ and $\bold x_-$ are points in $\Sigma\setminus\nu(\alpha\cup\beta)$.  We refer to $(\Sigma;\alpha,\bold x_+,\bold x_-)$ as a \emph{doubly pointed cut system}.
\end{definition}

The notion of slide-diffeomorphism introduced for cut systems and Heegaard tuples in Definition~\ref{def:cut} extends to the setting of doubly pointed cut systems if we require that handleslides are performed along arcs that are disjoint from the double points, isotopies of cut systems are supported away from the double points, and diffeomorphisms preserve the double points and their $(\pm)$--labeling.  For emphasis, we will sometimes refer to this equivalence relation as \emph{pointed-slide-diffeomorphism}.

\begin{definition}
	A \emph{doubly pointed trisection diagram} is a tuple $(\Sigma;\alpha,\beta,\gamma, \bold x_+,\bold x_-)$ where $\Sigma$ is a closed surface; each of $\alpha$, $\beta$, and $\gamma$ is a cut system for $\Sigma$; and each of $(\Sigma;\alpha,\beta,\bold x_+,\bold x_-)$, $(\Sigma;\beta,\gamma,\bold x_+,\bold x_-)$, and $(\Sigma;\gamma,\alpha,\bold x_+,\bold x_-)$ is pointed-slide-diffeomorphic to the doubly pointed Heegaard diagram shown as the top graphic of Figure~\ref{fig:Heeg_diags}, which we continue to refer to as the \emph{trivial $(g,k)$--diagram}, with the understanding that the double points are now included.
\end{definition}

A doubly pointed Heegaard diagram encodes a pair $(Y,K)$, where $Y$ is the closed 3--manifold described by the Heegaard diagram, and $K$ is a knot in $Y$ that intersects the Heegaard surface in the points $\bold x_\pm$ and intersects each of the handlebodies in single, boundary-parallel arc contained in the 0--handle of the handlebody.  For example, the trivial $(g,k)$--diagram encodes the pair $(Y_k,U)$, where $U$ is the unknot.

The requirement that each pair of doubly pointed cut systems in a doubly pointed trisection diagram be slide-diffeomorphic to a trivial diagram reflects the fact that, in a 1--bridge trisection, the 2--knot $\Kk$ intersects $X_i$ in a trivial disk $D_i$; i.e., $\partial D_i = U$ in $Y_k$.

\begin{proposition}\label{prop:realizing_2-knot}
	 A doubly pointed trisection diagram uniquely determines a 2--knot $(X,\Kk)$ with a 1--bridge trisection.
\end{proposition}

\begin{proof}
	The construction of the ambient manifold $X$ from the data of the trisection diagram was described in Section~\ref{sec:trisections}.  Let $H_\alpha$, $H_\beta$, and $H_\gamma$ denote the described handlebodies with boundary $\Sigma$, and let $X_1$, $X_2$, and $X_3$ denote the 4--dimensional pieces bounded by $H_\alpha\cup_\Sigma \overline H_\beta$, $H_\beta\cup_\Sigma \overline H_\gamma$, and $H_\gamma\cup_\Sigma \overline H_\alpha$, respectively.
	
	Let $\omega_\alpha$ be a properly embedded arc in $H_\alpha$, obtained by taking an arc $\omega_\alpha^*$ in $\Sigma\setminus\nu(\alpha)$ connecting $\bold x_+$ and $\bold x_-$ and perturbing its interior into $H_\alpha$.  Define $\omega_\beta$ and $\omega_\gamma$ similarly.  Note that these arcs are unique up to proper isotopy inside their respective handlebodies, since they are boundary-parallel.  Since $(\Sigma;\alpha,\beta;\omega_\alpha^*,\omega_\beta^*)$ is a standard doubly pointed Heegaard diagram (albeit, augmented with arcs connecting the double points), the union $\omega_\alpha\cup\overline\omega_\beta$ is an unknot in $\partial X_1 = H_\alpha\cup_\Sigma\overline H_\beta$.  The same is true of the pairs $(\partial X_2,\omega_\beta\cup\overline\omega_\gamma)$ and $(\partial X_3,\omega_\gamma\cup\overline\omega_\alpha)$.  These unknots can be capped off with disks $\Dd_i$ in a canonical way in the corresponding 4--dimensional 1--handlebodies $X_i$.  (See Lemma~2.3 of~\cite{MeiZup_Bridge-trisections_}.)
	
	Thus, the 2--knot $(X,\Kk)$, where $\Kk = \Dd_1\cup\Dd_2\cup\Dd_3$, is canonically determined by the original data of a doubly pointed trisection diagram.
\end{proof}

\subsection{Trisecting 2--knot exteriors}\label{subsec:exteriors}

In this subsection we discuss how a $1$--bridge trisection $X=X_1 \cup X_2 \cup X_3$ of a $2$--knot $(X,\Kk)$ immediately gives a relative trisection of the exterior $X \setminus \nu(K)$, and we use this to prove Theorem~\ref{thm:uniqueness}. We also show how to get a relative trisection diagram for this trisection of the exterior. To see the first fact, choose the neighborhood and its product structure $\nu(\Kk) \cong \Kk \times D^2$ so that each $X_i\cap\nu(\Kk)$ is $(X_i \cap \Kk) \times D^2$. Then we see that, in transitioning from $X$ to the exterior $X \setminus \nu(K)$, we have simply removed a 4--ball from each $X_i$, since $X_i\cap\Kk$ is a disk.  Similarly, we have removed a 3--ball from each $X_i \cap X_{i+1}$, and two disks from $X_1 \cap X_2 \cap X_3$; thus, the $X_i$ do in fact induce a trisection of the exterior.

\begin{proof}[Proof of Theorem~\ref{thm:uniqueness}]
 Consider two $1$--bridge trisections $X=X_1 \cup X_2 \cup X_3$ and $X=X'_1 \cup X'_2 \cup X'_3$ of the same $2$--knot $(X,\Kk)$. By an ambient isotopy we can assume that $X_i \cap \Kk = X'_i \cap \Kk$, since in both cases we simply have ``trisections'' of a $2$--sphere into three bigons meeting at two points. A further isotopy then arranges that the trisections agree on a neighborhood $\nu(\Kk)$ of $\Kk$. Thus, we have two relative trisections of the exterior $X \setminus \nu(\Kk)$ inducing the same (annular) open book on the boundary. Then Theorem~\ref{thm:relexistunique} tells us that these trisections become isotopic after stabilization, but since these stabilizations happen away from $\Kk$, this proves the result.
\end{proof}

As for the relative trisection diagram for $X \setminus \nu(\Kk)$, it should be clear now that this is simply the original doubly pointed diagram minus open disk neighborhoods of the two points. This is because the central surface $X_1 \cap X_2 \cap X_3$ has two disks removed, and yet every curve that bounded a disk in a given handlebody before removing $\nu(\Kk)$ still bounds a disk in the induced compressionbodies. To turn this into an arced relative diagram, perform handle slides on $\alpha$ and/or $\beta$ curves until the two boundary components are in the same component of $\Sigma \setminus (\alpha \cup \beta)$, and then connect them by an arc $\frak a = \frak b$. Then slide this arc over $\beta$ curves, and perhaps perform handle slides on $\gamma$, to get an arc disjoint from $\gamma$ and this is the arc $\frak c$. Repeating this process with $\gamma$ and $\alpha$ gets back to $\frak a^*$. This is illustrated in a few key examples in the next section.

Although the above argument is straightforward, it is worth pointing out two directions in which one might wish to extend these ideas both of which take more work. First, the given trisection {\em does not} induce a trisection of $\nu(\Kk)$; in particular, $X_1 \cap X_2 \cap X_3 \cap \nu(\Kk)$ is disconnected. This can be remedied by carefully stabilizing the given closed trisection inside $\nu(\Kk)$, so that we genuinely see the closed trisection as the union a trisection of $\nu(\Kk)$ to a trisection of $X \setminus \nu(\Kk)$. Secondly, if the bridge number is larger than one we do not get a trisection of either the neighborhood nor the exterior. This can also, in principle, be fixed in a similar way but can be quite complicated in practice. In particular, since higher genus surfaces never have $1$--bridge trisections, surgery along surfaces of nonzero genus, and along non-orientable surfaces, is necessarily more subtle. Kim and Miller show how to obtain a relative trisection for the exterior of a knotted surface from a (generalized) bridge trisection of the knotted surface~\cite{KimMil_Trisections-of-surface_18}; they apply their technique to study the Price twist, which can be described as a version of Gluck surgery for knotted projective planes.

\subsection{Examples of trisected 2--knot exteriors}\label{subsec:examples}

We now consider a number of important examples of 2--knots in simple 4--manifolds that can be put in 1--bridge position.  As discussed above, this will allow us to present annular arced trisection diagrams for these 2--knot exteriors, and we will make significant use of these in the proof of Theorem~\ref{thm:surgery} in Section~\ref{sec:surgery}.

Each image in the first column of Figure~\ref{fig:1b_Comps} is a doubly pointed trisection diagram $\DD$ corresponding to a 2--knot $(X,\Kk)$.  The corresponding image in the second column is the corresponding relative trisection diagram for the exterior $E_\Kk = X\setminus\nu(\Kk)$, as described above.  The corresponding image in the third column is an arced trisection diagram $\DD^\circ$ for $E_\Kk$.  We will now remark on each row.  Justification for the first column of Figure~\ref{fig:1b_Comps} is given by Figure~1 of~\cite{MeiZup_Bridge-trisections_}.

\begin{figure}[h!]
	\centering
	\includegraphics[width=.9\textwidth]{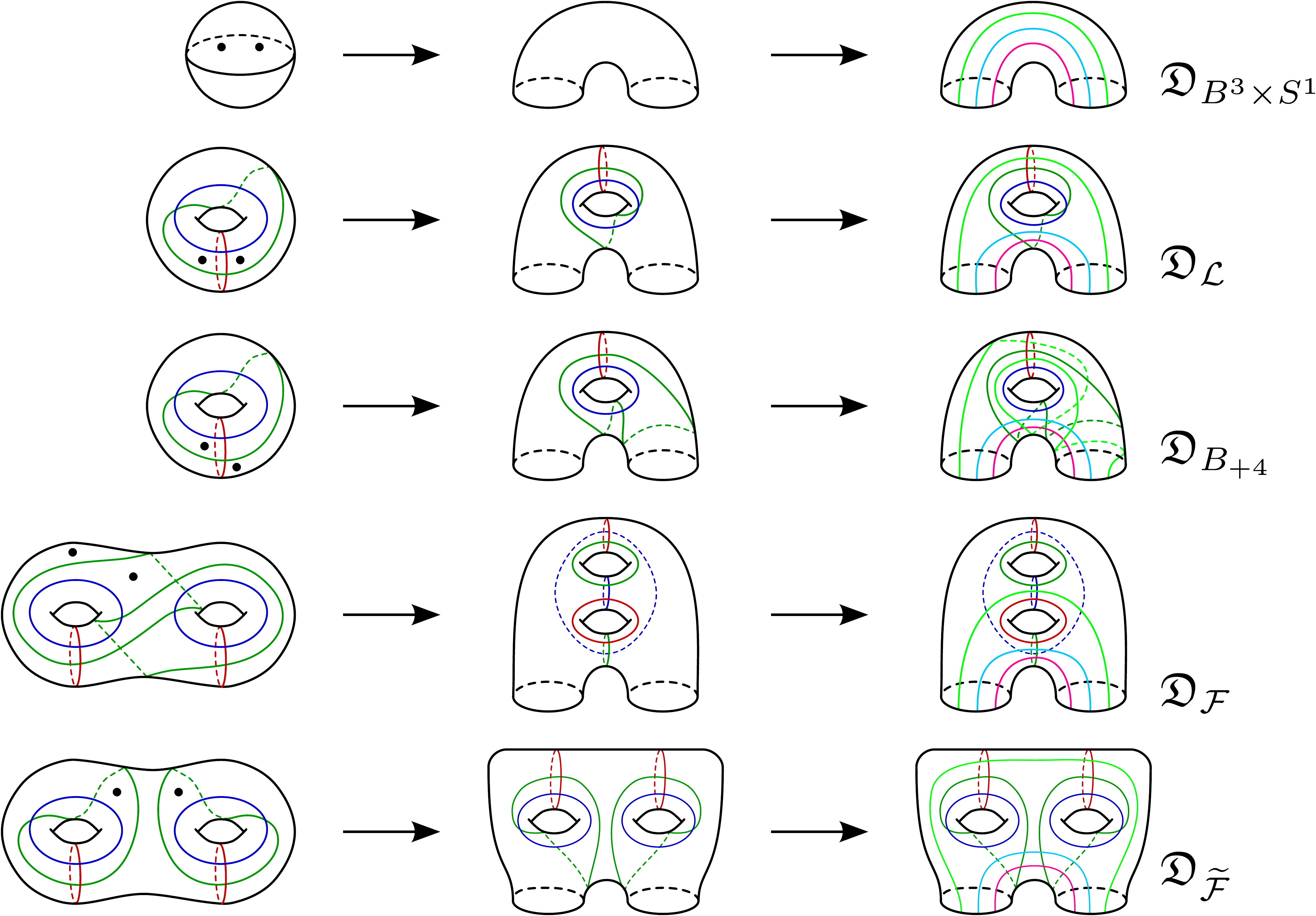}
	\caption{Five examples of how one can go from a doubly pointed trisection diagrams for a 2--knot $(X,\Kk)$ to an arced relative trisection diagrams for the exterior $E_\Kk$.  From top to bottom, the pairs are:  The unknotted 2--knot $(S^4,\Uu)$, projective line $(\CP^2,\Ll)$, the conic $(\CP^2,\Cc)$, the fiber $(S^2\times S^2,\Ff)$, and  the fiber $(S^2\wt\times S^2,\widetilde\Ff)$.  The corresponding complements are:  $B^3\times S^1$, $B^4$, the $\Z_2$--homology ball $B_{+4}$ with boundary $L(4,1)$, $S^2\times D^2$, and $S^2\times D^2$.}
	\label{fig:1b_Comps}
\end{figure}

The first row corresponds to the unknotted 2--knot in the 4--sphere: $(S^4,\Uu)$.  In this case, since $\Uu$ is fibered by 3--balls, $E_\Uu\cong B^3\times S^1$.  The arced trisection diagram $\DD^\circ_\Uu = \DD_{B^3\times S^1}$ is $0$--annular, since $\Uu\cdot\Uu = 0$.  Note that the corresponding trisection $\Tt^\circ_{B^3\times S^1}$ of $B^3\times S^1$ is the unique trisection of type $(0,0;2,0)$.

The second row corresponds to the projective line in the projective plane: $(\CP^2,\Ll)$.  In this case, $E_\Ll\cong B^4$, and the given arced trisection diagram $\DD^\circ_\Ll$ is $(+1)$--annular, since $\Ll\cdot\Ll = 1$.  Similarly, we can consider the mirror pair $(\overline\CP^2,\overline\Ll)$.  In this case, $E_{\overline\Ll}\cong B^4$, but $\DD^\circ_{\overline\Ll} = \overline\DD^\circ_\Ll$ is $(-1)$--annular, since $\overline\Ll\cdot\overline\Ll =-1$.

The third row corresponds to the degree two curve in the projective plane $(\CP^2,\Cc)$, which we refer to as the \emph{conic}.  In this case, $E_\Cc$ is the $\Z_2$--homology 4--ball $B_{+4}$ satisfying $\partial B_{+4} = L(4,1)$.  (See Subsection~\ref{subsec:surgery} for details.)  Since $\Cc\cdot\Cc = 4$, we have that $\DD^\circ_\Cc = \DD_{B_{+4}}$ is $(+4)$--annular.  Similarly, when considering the mirror pair $(\overline\CP^2,\overline\Cc)$, we have that $E_{\overline\Cc}\cong \overline B_{+4}$ and that $\DD^\circ_{\overline\Cc} = \overline\DD^\circ_\Cc = \DD_{B_{-4}}$ is $(-4)$--annular.

The trisections $\Tt^\circ_\Ll$ and $\Tt^\circ_\Cc$ are, up to mirroring, the only irreducible\footnote{A trisection is \emph{irreducible} if it is not the connected sum of trisections.} $(1,0;0,2)$--trisections~\cite{MeiZup_Bridge-trisections_}.  In general, the disk bundle $N_e$ with Euler number $e(N_e) = e$ admits a $(2,0;0,2)$--trisection~\cite{CasGayPin_Diagrams-for-relative_18}, and it is not known whether or not there are other 4--manifolds that admit irreducible trisections of this sort.

The fourth and fifth rows correspond, respectively, to the fiber $\Ff = S^2\times\{pt\}$ in the trivial bundle
$$S^2\hookrightarrow S^2\times S^2\twoheadrightarrow S^2$$
and the fiber $\widetilde\Ff = \pi^{-1}(pt)$ in the twisted bundle
$$S^2\hookrightarrow S^2\wt\times S^2\stackrel{\pi}{\twoheadrightarrow} S^2.$$
In these cases, we have $E_\Ff \cong E_{\widetilde\Ff} \cong S^2\times D^2$, and both of the arced diagrams $\DD^\circ_\Ff=\DD_\Ff$ and $\DD^\circ_{\widetilde\Ff} = \DD_{\widetilde\Ff}$ are $0$--annular $(2,0;0,2)$--trisection diagrams.

\subsection{Recording the core of $S^2\times D^2$}\label{subsec:core}

We end this section with a discussion of the special case of a 2--sphere inside a 4--manifold with nonempty boundary, which is not technically a 2--knot, by our definition.  By a \emph{core}, we mean the 2--sphere $S^2\times\{pt\}$ inside $S^2\times D^2$.  The importance of the core is 2--fold:  First, given any 2--knot $(X,\Kk)$, inside the neighborhood $\nu(\Kk)$, the 2--knot $\Kk$ is a core.  Second, when we form a new 4--manifold by gluing $S^2\times D^2$ to a 4--manifold $E$ with $\partial E = S^2\times S^1$, the core naturally into the result as the 2--knot $(E\cup S^2\times D^2, S^2\times\{pt\})$, which is called the \emph{dual 2--knot} for this gluing.

Rather than explicitly develop a theory of bridge trisected position with respect to relative trisections for knotted surfaces in 4--manifolds with boundary, we will content ourselves in the present article to an \emph{ad hoc} description of a core, which we will refer to not as a doubly pointed relative trisection diagram (as one might), but simply as a \emph{core diagram}, two of which are shown in Figure~\ref{fig:core_knots}.

\begin{definition}
	A \emph{core diagram} is a tuple
	$$(\Sigma;\alpha,\beta,\gamma;\frak a,\frak b,\frak c, \bold x_+, \bold x_-),$$
	where
	$$(\Sigma;\alpha,\beta,\gamma;\frak a,\frak b,\frak c)$$
	is a 0--annular arced relative trisection diagram	for $S^2\times D^2$, and $\bold x_+$ and $\bold x_-$ are points in the exterior of the arcs and curves in $\Sigma$ such that $\bold x_+$ and $\bold x_-$ can be isotoped, in the exterior of the arcs and the curves, to lie in distinct components of $\partial\Sigma$.
\end{definition}

\begin{figure}[h!]
	\centering
	\includegraphics[width=.5\textwidth]{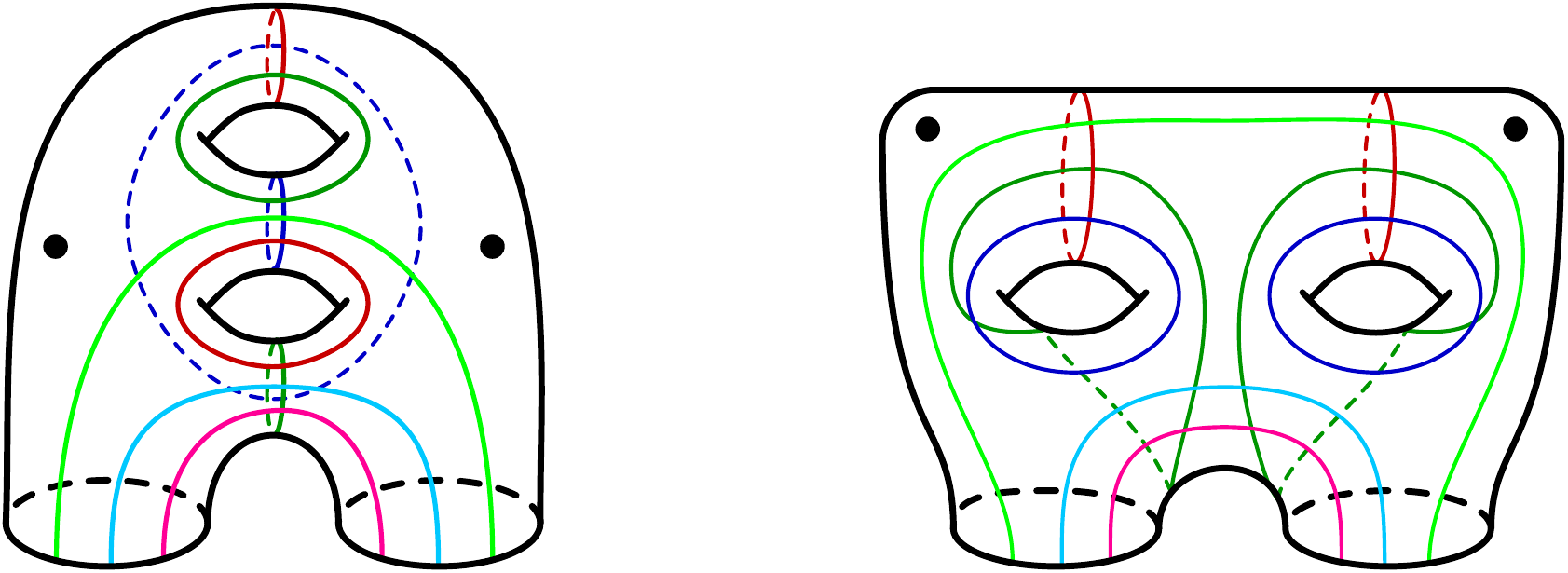}
	\caption{Two core diagrams, each of which represents the pair $(S^2\times D^2,S^2\times\{pt\})$.}
	\label{fig:core_knots}
\end{figure}

The next lemma verifies that the result of gluing a core diagram to a compatible arced diagram is a doubly pointed diagram for the resulting dual 2--knot.

\begin{lemma}
\label{lem:cores}
	If $\DD_\text{core}$ is a core diagram with underlying arced diagram $\DD$ describing $S^2\times D^2$  and $\DD^\circ$ is a 0--annular arced diagram for a 4--manifold $E$ with $\partial E \cong S^2\times S^1$, then $\DD^\circ\cup\DD_\text{core}$ is a doubly pointed trisection diagram for the pair $(X,\Kk)$, where $X = E\cup_{\Upsilon(\DD,\DD^\circ)} S^2\times D^2$ and $\Kk$ is the dual 2--knot for this gluing.
\end{lemma}

\begin{proof}
	By Lemma~\ref{lem:compatible}, $\DD^\circ$ and $\DD$ are gluing compatible, so their union describes the desired ambient 4--manifold.  It remains to see that the inclusion of $\bold x_+$ and $\bold x_-$ into $\DD^\circ\cup\DD$ changes this trisection diagram into a doubly pointed trisection diagram.
	
	The points $\bold x_+$ and $\bold x_-$ can be connected by arcs that are parallel to the arcs $\DD$, because of the way they can be assumed to lie in $\partial\DD$.  It follows that these points encode a sphere $S^2\times\{pt\}$, where $\{pt\}$ can be assumed to lie in the boundary $ S^2\times S^1$, if desired.  We already have that $\bold x_\pm$ lie in $S^2\times S^1$.  The arcs connecting them naturally lie on annular fibers in $S^2\times S^1$ and are pairwise parallel there.  If follows that there is a sphere in $S^2\times S^1$ that is cut into three disks (the traces of these parallelisms) that intersect pairwise in the arcs and have common intersection $\bold x_\pm$.  Since this sphere intersects each annulus page in a properly-embedded arc, it intersects any circle core of a page in a point.  It follows that the sphere is $S^2\times\{pt\}$.
	
	This sphere, along with the induced 1--bridge trisection just described with respect to the twice-punctured genus two central surface of the trisection of $S^2\times D^2$ corresponding to $\DD$, persists in the trisection corresponding to $\DD_\text{core}\cup\DD^\circ$, giving an honest 1--bridge trisection of $(X,\Kk)$, as desired.
\end{proof}

\section{Surgeries on 2--knots and their diagrams}\label{sec:surgery}

In this section, we give formal definitions of the relevant surgery operations, and we prove Theorem~\ref{thm:surgery}, which gives a trisection diagrammatic illustration of each surgery operation.

\subsection{Surgery on 2--knots}\label{subsec:surgery}

Let $\Kk$ be a 2--knot in a 4--manifold $X$ with self-intersection zero; i.e., $\Kk$ has a trivial normal bundle in $X$.  Then, $\nu(\Kk)\cong S^2\times D^2$, and the exterior $E_\Kk$ of $\Kk$ in $X$ has boundary $\partial E_\Kk\cong S^2\times S^1$.  There are two important surgery operations on such 2--knots.  First, consider the 4--manifold
$$X(\Kk) = E_\Kk\cup_{S^2\times S^1}B^3\times S^1$$
obtained by gluing $B^3\times S^1$ to $E_\Kk$ using some diffeomorphism of their $S^2\times S^1$ boundaries.  The manifold $X(\Kk)$ is said to be the result of performing \emph{sphere surgery} on $\Kk$ in $X$.  By~\cite{LauPoe_A-note-on-4-dimensional_72}, every self-diffeomorphism of $S^2\times S^1$ extends over $B^3\times S^1$, so the result of sphere surgery depends only on $\Kk$ and $X$.

As an example of sphere surgery, note that if $\Kk$ is a fibered 2--knot in $S^4$ with fiber some punctured 3--manifold $M^\circ$, then $X(\Kk)$ has the structure of a 3--manifold bundle over $S^1$ with fiber $M$.  In particular, since the unknotted 2--sphere $\Uu$ in $S^4$ is fibered by 3--balls, we have that $S^4(\Uu)\cong S^3\times S^1$.

Next, consider the 4--manifold
$$X_*(\Kk) = E_\Kk\cup_\tau S^2\times D^2$$
obtained by gluing $S^2\times D^2$ to $E_\Kk$ using the diffeomorphism $\tau\colon S^2\times S^1\to S^2\times S^1$ defined by
$$\tau(x,\theta) = (\rho_\theta(x),\theta),$$
where $\rho_\theta\colon S^2\times S^2$ is rotation (about some axis) through the angle $\theta$.  This operation was introduced by Gluck~\cite{Glu_The-embedding-of-two-spheres_62}, who showed that the diffeomorphism $\tau$ (called the \emph{Gluck twist}) is the unique diffeomorphism of $S^2\times S^1$ that does not extend to a diffeomorphism of $S^2\times D^2$ and that the result of this gluing is independent of the choice of axis of rotation.  The manifold $X_*(\Kk)$ is said to be the result of \emph{Gluck surgery} on $\Kk$ in $X$.

As an example of Gluck surgery, consider the fiber $\Kk = S^2\times\{pt\}$ inside $X = S^2\times S^2$.  Then, $(X)_*(\Kk)\cong S^2\wt\times S^2$. On the other hand, Gluck showed that $S^4_*(\Kk)$ is a homotopy 4--sphere for all $\Kk\subset S^4$, but it is unknown whether or not this homotopy 4--sphere is always diffeomorphic to $S^4$~\cite{Glu_The-embedding-of-two-spheres_62}.

There is an inverse operation to sphere surgery, which we call circle surgery.  Let $\omega$ be a  simple, closed curve in a 4--manifold $X$.  Then, $\nu(\omega)\cong B^3\times S^1$, and the exterior $E_\omega = X\setminus\nu(B^3\times S^1)$ has boundary $S^2\times S^1$. We can form a new manifold
$$X_\varepsilon(\omega) = E_\omega\cup_\varepsilon S^2\times D^2,$$
where $\varepsilon:S^2\times S^1 \to S^2 \times S^1$ is either $\id$ or $\tau$. In either case, we call $X_\varepsilon(\omega)$ the result of \emph{circle surgery} on $\omega$ in $X$.  Note that $X_\tau(\omega)$ and $X_\id(\omega)$ are related by Gluck surgery along their cores.

Now we turn to the instance in which $\Kk$ has self-intersection $\pm1 $ or $\pm4$ inside $X$.  First, we assume the former case, so $\partial E_\Kk\cong S^3$.  Consider the 4--manifold
$$X_{\pm 1}(\Kk) = E_\Kk\cup_{S^3}B^4$$
obtained by gluing $B^4$ to $E_\Kk$ using some diffeomorphism of $S^3$.  As in the case of sphere surgery, the result is independent of the diffeomorphism.  The manifold $X_{\pm 1}(\Kk)$ is said to be the result of a \emph{$(\pm 1)$--blowdown} of $\Kk$ in $X$.

When $\Kk$ has self-intersection $\pm 4$, we have that $\partial E_\Kk\cong L(4,\pm 1)$.  Fintushel and Stern showed that $L(4,\pm 1)$ bounds a rational-homology 4--ball $B_{\pm 4}$ and that every self-diffeomorphism of $L(4,\pm 1)$ extends over $B_{\pm 4}$~\cite{FinSte_Rational-blowdowns_97}.  Thus, the 4--manifold
$$X_{\pm 4}(\Kk) = E_\Kk\cup_{L(4,\pm 1)}B_{\mp 4},$$
which is said to be the result of a \emph{$(\pm 4)$--rational blowdown} along $\Kk$ in $X$, depends only on $\Kk$ and~$X$.  Note that $L(4,1)$ does not admit an orientation reversing self-diffeomorphism, so we must be careful: a $(\pm 4)$--rational blowdown on a $(\pm 4)$--framed 2--knot $\Kk$ in $X$ requires gluing $B_{\mp 4}$ to $E_\Kk$ using an orientation-reversing diffeomorphism from $L(4,\pm 1)$ to $L(4,\mp 1)$.

The simplest example of a 2--knot with self-intersection $+4$ is the degree-two curve $\Cc = \Cc_2$ in $\CP^2$.  In fact, $\CP^2\setminus\nu(\Cc)\cong B_{+4}$~\cite{FinSte_Rational-blowdowns_97}.  Similarly, $\overline \CP^2\setminus\nu(\overline\Cc)\cong B_{-4}$.  So, the result of a $(+4)$--rational blowdown on $\Cc_2$ in $\CP^2$ can be described as
$$\CP^2_{+4}(\Cc) \cong B_{+4}\cup_{L(4,1)}B_{-4},$$
which turns out to coincide with the spun manifold $\Ss(\RP^3)$, as well as the non-trivial sphere-bundle over $\RP^2$. See Example~\ref{ex:conic}, below.

\subsection{Gluck twins}
\label{subsec:twins}

Although the main goal of this section (and this paper) is to describe how to go from a doubly pointed trisection diagram for a 2--knot $(X,\Kk)$ to one of several manifolds resulting from performing some surgery operation along $\Kk$ in $X$, the methods employed are slightly more general, applying to any 4--manifold $E$ with $\partial E = S^2\times S^1$.  This is, however, something of a distinction without a difference, since any such $E$ gives rise to a pair of 2--knots:
$$(X,\Kk) = (E\cup_\id S^2\times D^2,S^2\times\{pt\}) \text{ and } (X',\Kk') = (E\cup_\tau S^2\times D^2, S^2\times\{pt\}),$$
where $X$ and $X'$ may or may not be diffeomorphic, homeomorphic, or even homotopy-equivalent, and even if they are diffeomorphic, $\Kk$ and $\Kk'$ may not be smoothly or even topologically isotopic.  We call the 2--knots $(X,\Kk)$ and $(X',\Kk')$ \emph{Gluck twins}, since they share a common exterior.

We now mention a handful of interesting examples of Gluck twins. (See also Section~\ref{sec:examples}, where we consider these examples -- and others -- through the lens of trisections.) First, the fibers $(S^2\times S^2,\Ff)$ and $(S^2\wt\times S^2, \widetilde\Ff)$ are twins, since, as we saw in Subsection~\ref{subsec:examples}, they both have exterior $S^2\times D^2$.  This is an example where the ambient space of the twins are not even homotopy-equivalent.  To the other extreme, there are many families of 2--knots in $S^4$ that are known to be smoothly isotopic to their twin.  (In particular, the twins have the same ambient space.)  Gluck gave a sufficient condition, based on a Seifert hyper-surface bounded by the 2--knot, for a 2--knot to be equivalent to its twin; hence, the ambient 4--manifold is preserved, as well. (See Theorem~17.1 of~\cite{Glu_The-embedding-of-two-spheres_62}.)  This condition is strong enough to show that ribbon 2--knots and fibered, homotopy-ribbon 2--knots in $S^4$ are ambient isotopic to their twins~\cite{Coc_Ribbon-knots_83}.

In between this extreme, there are many interesting examples.  See~\cite{Sun_Surfaces-in-4-manifolds:_15} for a discussion of the case of 2--knots that are topologically isotopic within a 4--manifold, but not smoothly isotopic.  In a slightly different vein, Gordon showed that 2--knots in $S^4$ are not determined by their complement, in general~\cite{Gor_Knots-in-the-4-sphere_76}.

%

\subsection{Proof of Theorem~\ref{thm:surgery}}
\label{subsec:proof_surgery}

We are now ready to prove our main result, which gives diagrammatic characterizations of 4--dimensional surgery operations on 2--knots. The proof will follow from a sequence of lemmata.  First, we consider the case corresponding to sphere surgery on a 2--knot.  We refer the reader to Figure~\ref{fig:surgery_sequence} for an illustration of the statement of our first lemma.  Let $\DD_{B^3\times S^1}$ be the 0--annular arced trisection diagram corresponding to $B^3\times S^1$, as in Figure~\ref{fig:thm_gluings} and Figure~\ref{fig:1b_Comps}.

\begin{lemma}
\label{lem:surgery}
	Let $\DD^\circ$ be a 0--annular arced trisection diagram for a 4--manifold $X$ with $S^2\times S^1$ boundary. Then, the union
	$$\DD' = \DD^\circ\cup\DD_{B^3\times S^1}$$
	is a trisection diagram for the 4--manifold $X' = X\cup (B^3\times S^1)$.
\end{lemma}

Note that $\overline\DD_{B^3\times S^1} = \DD_{B^3\times S^1}$, in accordance with the fact that $B^3\times S^1$ admits an orientation-reversing self-diffeomorphism.

\begin{proof}
	By Lemma~\ref{lem:compatible}, since $\DD^\circ$ and $\DD_{B^3\times S^1}$ are both 0--annular arced trisection diagrams, they are gluing compatible.  So, the diagram $\DD'$ describes $X'$ by Proposition~\ref{prop:glue}.  Note that this gluing is unique by~\cite{LauPoe_A-note-on-4-dimensional_72}.
\end{proof}

\begin{figure}[h!]
	\centering
	\includegraphics[width=.75\textwidth]{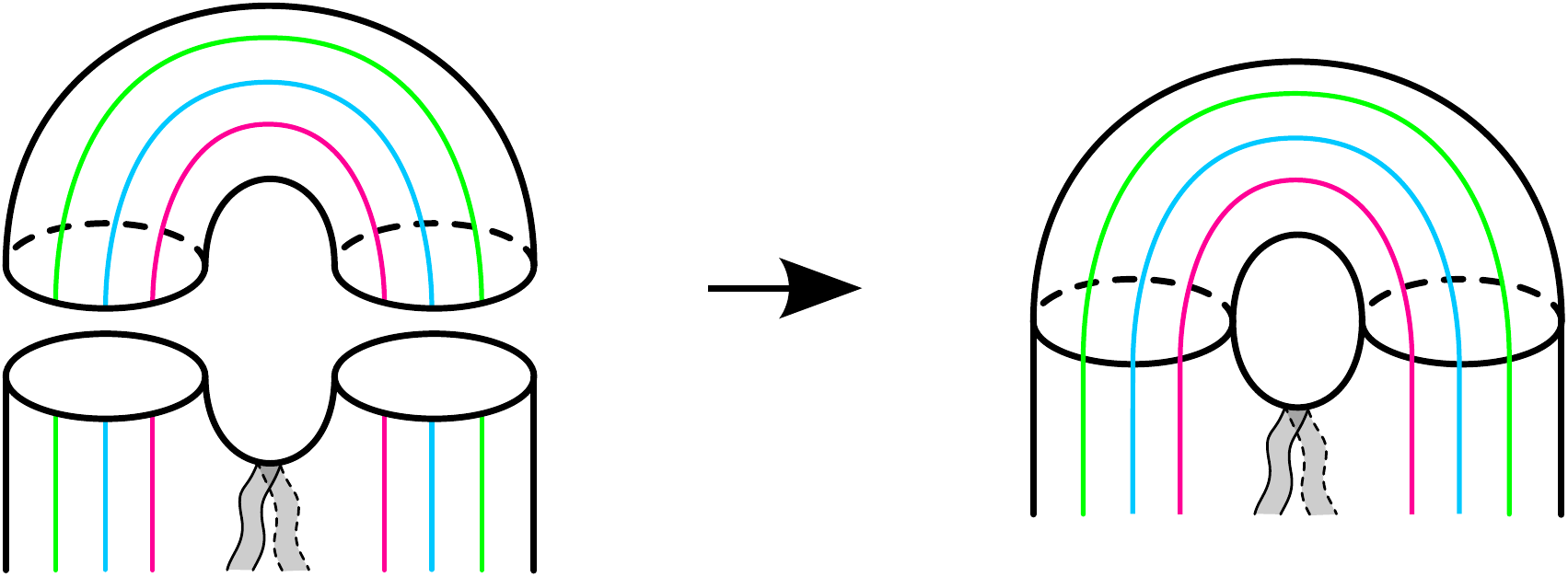}
	\caption{The left image shows 0--annular arced trisection diagrams $\DD^\circ$ and $\DD_{B^3\times S^1}$, corresponding respectively to a 4--manifold $X$ with $\partial X = S^2\times S^1$ and the 4--manifold $B^3\times S^1$, as in Lemma~\ref{lem:surgery}.  The right image shows the result $\DD'$ of gluing these diagrams together.}
	\label{fig:surgery_sequence}
\end{figure}

Next, we consider the case corresponding to a rational blowdown on a 2--knot with self-intersection $\pm 4$.  We refer the reader to Figure~\ref{fig:rational_sequence} for an illustration of the second lemma.  The proof is nearly identical to that of the previous lemma. Let $\DD_{B_{+4}}$ be the $(+4)$--annular arced trisection diagram corresponding to $B_{+4}$, as in Figures~\ref{fig:thm_gluings} and~\ref{fig:1b_Comps}.

\begin{lemma}
\label{lem:rational}
	Let $\DD^\circ$ be a $(+4)$--annular arced trisection diagram for a 4--manifold $X$ with $L(4,1)$ boundary.  Then, the union
	$$\DD' = \DD^\circ\cup\overline\DD_{B_{+4}}$$
	is a trisection diagram for the 4--manifold $X\cup \overline{B_{+4}}$.  Similarly, $\partial X = L(4,-1)$ boundary and $\DD^\circ$ is $(-4)$--annular, then the diagram
	$$\DD' = \DD^\circ\cup\DD_{B_{+4}}$$
	describes the 4--manifold $X\cup B_{+4}$.
\end{lemma}

\begin{proof}
	Suppose $\partial X = L(4,1)$. Since $\DD^\circ$ and $\overline\DD_{B_{+4}} = \DD_{B_{-4}}$ are both $(+4)$--annular arced trisection diagram, they are gluing compatible by Lemma~\ref{lem:compatible}.  So, the diagram $\DD'$ describes $X$ by Proposition~\ref{prop:glue}.  Similarly, if $\partial X = L(4,-1)$, then $\DD^\circ\cup\DD_{B_{+4}}$ describes $X\cup B_{+4}$, as desired.  Note that these gluings are unique by~\cite{FinSte_Rational-blowdowns_97}.
\end{proof}

\begin{figure}[h!]
	\centering
	\includegraphics[width=.75\textwidth]{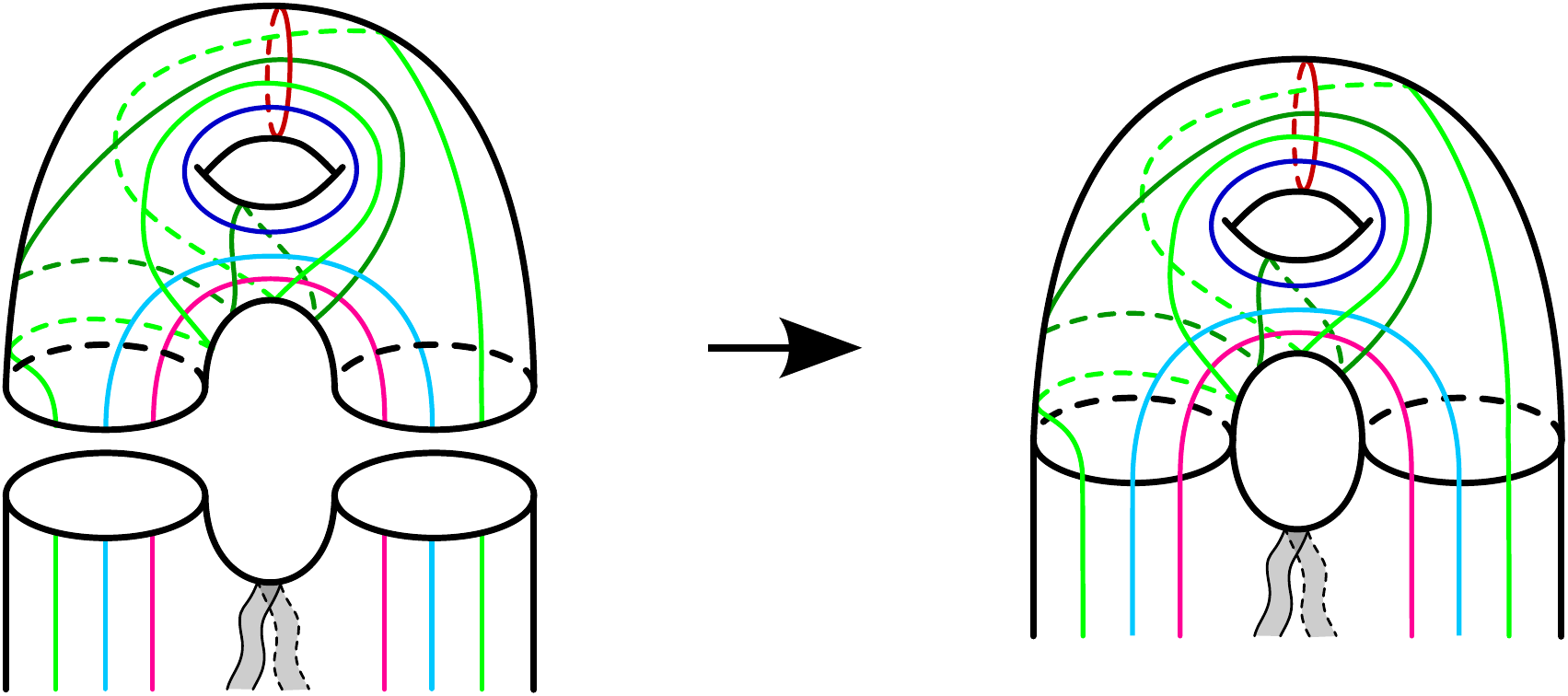}
	\caption{The left two images are the $(+4)$--annular arced trisection diagrams $\DD^\circ$ and $\DD_{B_{-4}}$, corresponding respectively to a 4--manifold $X$ with $\partial X = L(4,1)$ and the rational homology 4--ball $B_{-4}$, as in Lemma~\ref{lem:rational}.  The right image shows the trisection diagram $\DD'$ obtained by gluing these diagrams together, which describes the 4--manifold $X\cup B_{-4}$.}
	\label{fig:rational_sequence}
\end{figure}

Next, we consider the case corresponding to a blowdown on a 2--knot with self-intersection $\pm 1$. In order to state the next two lemmata, we need to introduce new objects, which we call \emph{twisted annuli} (more precisely, \emph{$\frak a$--twisted}, \emph{$\frak b$--twisted}, or \emph{$\frak c$--twisted annuli}, as the case may be), and which act as arced trisection diagrams, though they are not. The twisted annuli are denoted $\DD_\frak a$, $\DD_\frak b$, and $\DD_\frak c$ and are shown in Figure~\ref{fig:twisteds}.

\begin{figure}[h!]
	\centering
	\includegraphics[width=.75\textwidth]{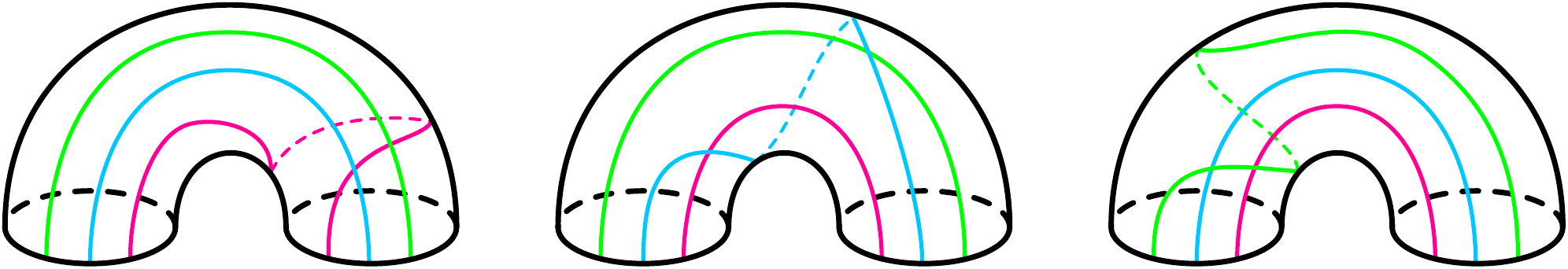}
	\caption{From left to right, the $\frak a$--, $\frak b$--, and $\frak c$--twisted annuli $\DD_\frak a$, $\DD_\frak b$, and $\DD_\frak c$, which are \emph{not} arced trisection diagrams.}
	\label{fig:twisteds}
\end{figure}

An arced trisection diagram has the property that the the red-arcs are slide-equivalent to the blue-arcs, which are slide-equivalent to the green-arcs.  Thus, an arced trisection diagram on the annulus must contain three parallel arcs, one of each color; i.e., it must be the arced trisection diagram $\DD_{B^3\times S^1}$ for $B^3\times S^1$.  Although the twisted annuli are not arced trisection diagrams, they will feature in our diagrammatic results as though they were.
We refer the reader to Figure~\ref{fig:blowdown_sequence} for an illustration of our third lemma, which corresponds to the case of a blowdown on a 2--knot with self-intersection $\pm 1$.

\begin{lemma}
\label{lem:blowdown}
	Let $\DD^\circ$ be a $(+1)$--annular arced trisection diagram for a 4--manifold $X$ with $\partial X = S^3$. Let $\DD_\frak a$ be the $\frak a$--twisted annulus.  Then, the union
	$$\DD' = \DD^\circ\cup\DD_\frak a$$
	is a trisection diagram for 4--manifold $X\cup B^4$.  Similarly, if $\DD^\circ$ is $(-1)$--annular, then the diagram
	$$\DD' = \DD^\circ\cup\overline\DD_\frak a$$
	describes $X\cup B^4$.
\end{lemma}

\begin{proof}
	Suppose $\DD^\circ$ is a $(+1)$--annular arced trisection diagram. Let $\overline\DD_{\Ll}$ denote the toroidal, $(-1)$--annular arced trisection diagram for $B^4$ shown in the top left of Figure~\ref{fig:blowdown_sequence}.  By Lemma~\ref{lem:compatible}, $\DD^\circ$ and $\overline\DD_{\Ll}$ are gluing compatible.  So, the diagram $\DD''$ shown in the top middle of Figure~\ref{fig:blowdown_sequence} describes $X\cup B^4$.

\begin{figure}[h!]
	\centering
	\includegraphics[width=.9\textwidth]{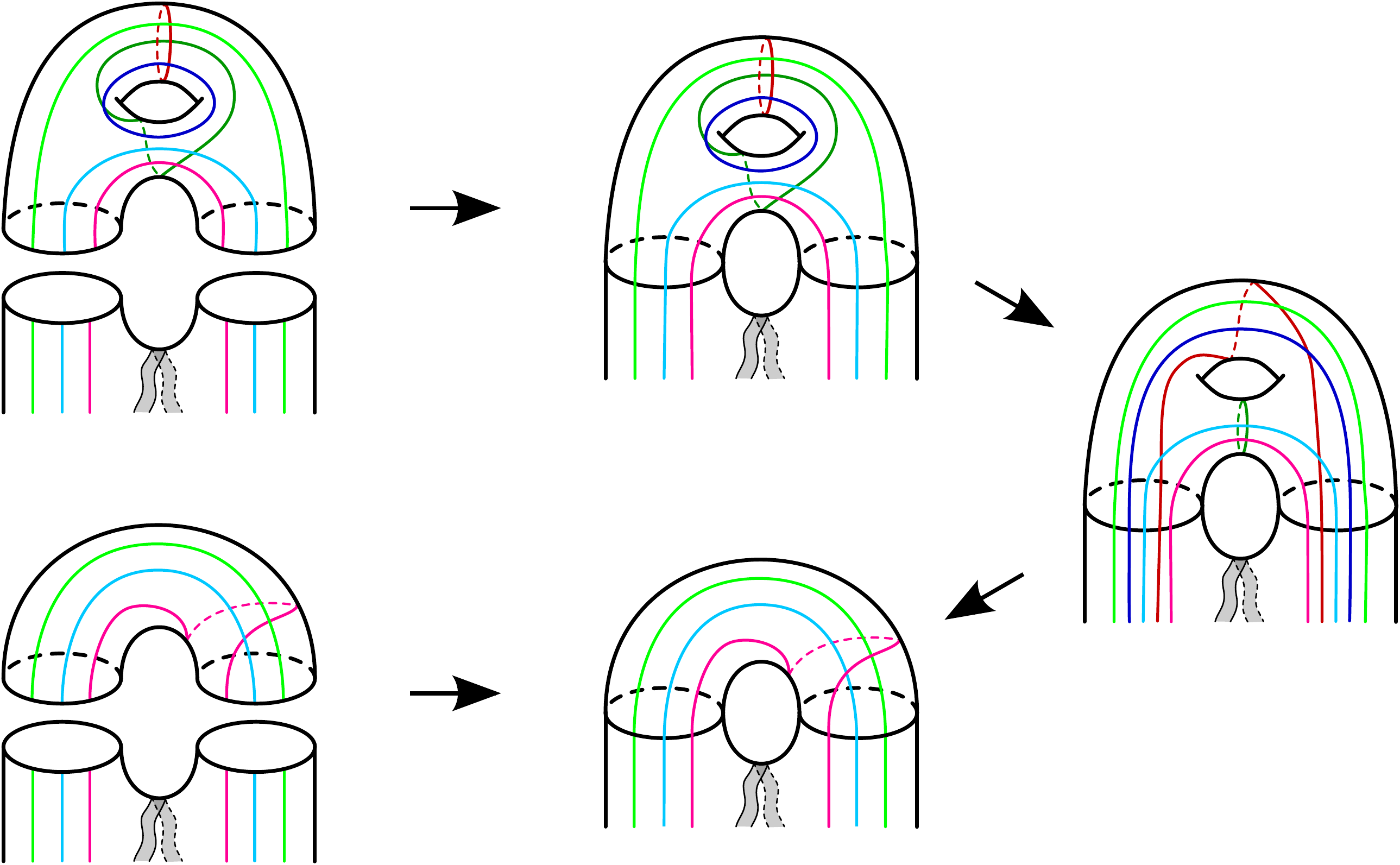}
	\caption{The top two images are the arced trisection diagrams $\DD^\circ$ and $\overline\DD_{\Ll}$, corresponding respectively to a 4--manifold $X$ with $\partial X\cong S^3$ and  $B^4$, and the trisection diagram $\DD''$ that results from their gluing. Similarly, the bottom two figures show the same thing, except with $\DD_\frak a$ replacing $\overline\DD_\Ll$.  The right-most trisection diagram is slide-diffeomorphic to the top middle one via a single Dehn twist about the central hole and two slides and the destabilizes to the bottom middle one.}
	\label{fig:blowdown_sequence}
\end{figure}

	Now, since $\DD^\circ$ is, by assumption, an annular arced trisection diagram we can assume that, away from the local pictures show in Figure~\ref{fig:blowdown_sequence}, the pink arc/curve is parallel to the light blue arc/curve.  Our goal is to destabilize $\DD''$ using this common $\alpha/\beta$--curve.  First, we do three things.  Modify $\DD''$ by performing a left-handed Dehn twist along the dark blue curve.  This has the effect of ``straightening out'' the dark green curve, and ``twisting'' the dark red curve.  Next, slide the dark blue curve over the light blue curve.  This will cause it to exit the local picture.  Finally, slide the dark red curve over the pink curve.  The end result is the diagram shown on the right side of Figure~\ref{fig:blowdown_sequence}.  This diagram can be destabilized by surgering the surface along the dark green curve and deleting the light blue curve and the pink curve, which were parallel.  The resulting diagram, after lightening the shades of the remaining red and blue curves, is shown in the bottom middle of the figure.  Note however, that this diagram is precisely the desired diagram $\DD'$, completing the proof in the case of self-intersection~+1.
	
	  The same proof works for the case that $\DD^\circ$ is $(-1)$--annular, except that the role of $\overline\DD_{\Ll}$ is played by $\DD_{\Ll}$ and the role of $\DD_\frak a$ is played by $\overline\DD_\frak a$. (The left-handed Dehn twist used in the modification of $\DD''$ will be right-handed, this time.)
\end{proof}

\begin{remark}
\label{rmk:twisted_blowdown}
	Lemma~\ref{lem:blowdown} is also valid if the $\frak c$--twisted annulus is used instead of the $\frak a$--twisted annulus.  More precisely, the lemma holds if every instance of $\DD_\frak a$ (respectively, $\overline\DD_\frak a$) is replaced with $\overline\DD_\frak c$ (respectively, $\DD_\frak c$).  There are a number of ways to see this.  One is that there is an orientation-reversing diffeomorphism of $B^4$ that has the effect of interchanging the red and green arcs and curves in the diagrams $\overline\DD_{\Ll}$ and $\DD_{\Ll}$.  Alternatively, in the proof of the lemma, the destabilization could be done using the dark red curve and the light green and light blue arcs, instead of the dark green curve and the pink and light blue arcs.
	
	We do not have a rendering of Lemma~\ref{lem:blowdown} that involves $\DD_\frak b$; this is an artifact of the asymmetry in the definition of an arced trisection diagram -- namely, that the $\frak a$--arcs cannot be made to coincide with the $\frak c$--arcs when the monodromy of the induced open book decomposition is not the identity.
\end{remark}

Finally, we arrive at the most complicated case: that corresponding to Gluck surgery.  In this case, we will make use of two different gluings of arced trisection diagrams to account for the two possible gluings of $S^2\times D^2$ to a 4--manifold $X$ with $\partial X = S^2\times S^1$. We refer the reader to Figures~\ref{fig:identity_sequence} and~\ref{fig:gluck_sequence} for illustrations of our fourth lemma.  In this lemma, $\DD_\text{caps}$ diagram consisting of two disjoint disks, drawn as caps, as in the right-most graphic in Figure~\ref{fig:identity_sequence}.  Note that $\DD_\text{caps}$ is not an honest trisection diagram of any sort.

\begin{lemma}
\label{lem:gluck}
	Let $\DD^\circ$ be a $0$--annular arced trisection diagram for a 4--manifold $X$ with $\partial X = S^2\times S^1$.  Then, the unions
	$$\DD' = \DD^\circ\cup\DD_\text{caps},$$
	and
	$$\DD'' =\DD^\circ\cup\DD_\frak a$$
	are trisection diagrams for 4--manifolds $X'$ and $X''$ satisfying
	$$\{X',X''\} = \{X\cup_\id S^2\times D^2, X\cup_\tau S^2\times D^2\}.$$
\end{lemma}

\begin{proof}
	Consider the arced trisection diagrams $\DD_\Ff$ and $\DD_{\widetilde\Ff}$ corresponding to the exteriors of the fiber 2--knots $\Ff$ and $\widetilde\Ff$ in  $S^2\times S^2$ and $S^2\wt\times S^2$, respectively. These diagrams are shown in the top left of Figures~\ref{fig:identity_sequence} and~\ref{fig:gluck_sequence}, respectively.  As $0$--annular arced trisection diagrams,  $\DD^\circ$, $\DD_\Ff$, and $\DD_{\widetilde\Ff}$ are pairwise gluing compatible by Lemma~\ref{lem:compatible}.  There are three things to be shown: that gluing on $\DD_\text{caps}$ corresponds to gluing on $\DD_\Ff$, that gluing on $\DD_\frak a$ corresponds to gluing on $\DD_{\widetilde\Ff}$, and that the gluing diffeomorphisms $\Upsilon(\DD^\circ,\DD_\Ff)$ and $\Upsilon(\DD^\circ,\DD_{\widetilde\Ff})$ are distinct in the mapping class group $\Mm(S^2\times S^1)$.  The statements of the first two claims are illustrated by Figures~\ref{fig:identity_sequence} and~\ref{fig:gluck_sequence}, and their proofs are exhibited in Figures~\ref{fig:identity_moves} and~\ref{fig:gluck_moves}.

\begin{figure}[h!]
	\centering
	\includegraphics[width=.9\textwidth]{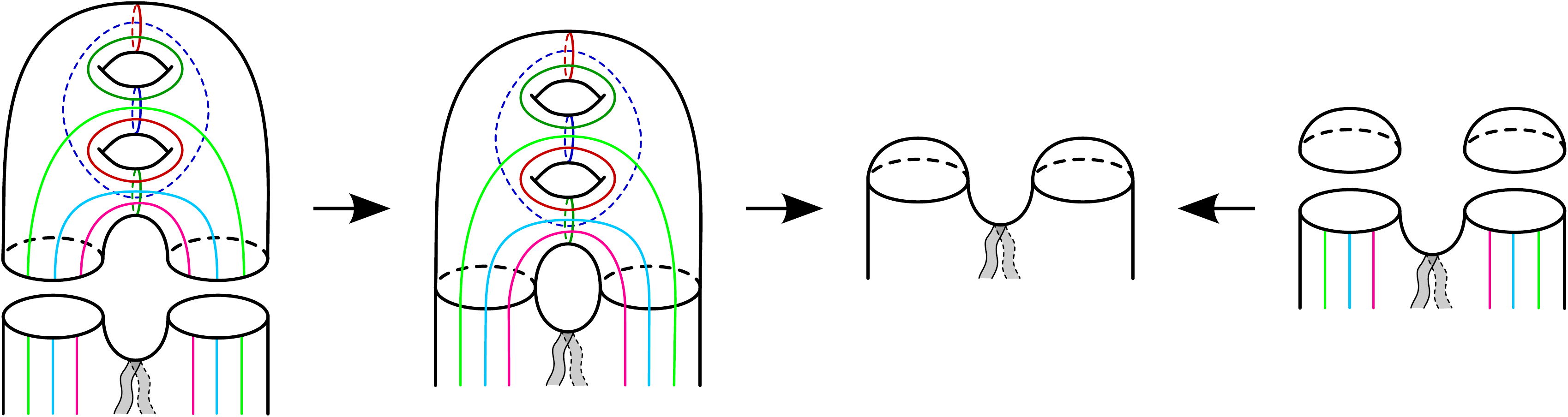}
	\caption{The left two images are the arced trisection diagrams $\DD^\circ$ and $\DD_\Ff$, corresponding respectively to a 4--manifold $X$ with $\partial X = S^2\times S^1$ and the exterior $S^2\times D^2$ of the fiber 2--knot $\Ff$ in $S^2\times S^2$, and the result of their gluing.  Similarly, the right two figures show the same thing, except with $\DD_\text{caps}$ replacing $\DD_\Ff$.  The equivalence of the middle arrow is the content of Figure~\ref{fig:identity_moves}.}
	\label{fig:identity_sequence}
\end{figure}

\begin{figure}[h!]
	\centering
	\includegraphics[width=.9\textwidth]{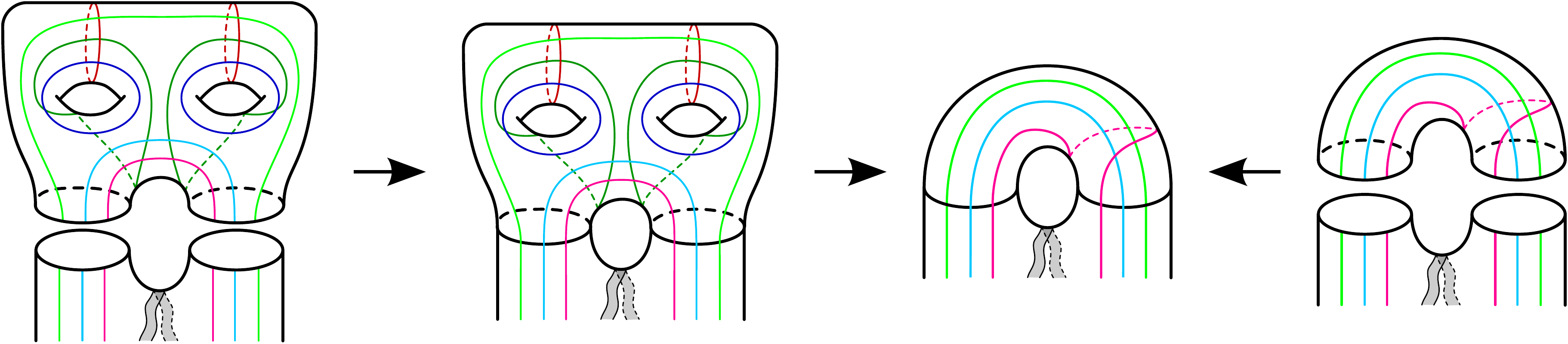}
	\caption{The left two images are the arced trisection diagrams $\DD^\circ$ and $\DD_{\widetilde\Ff}$, corresponding respectively to a 4--manifold $X$ with $\partial X = S^2\times S^1$ and the exterior $S^2\times D^2$ of the fiber 2--knot $\widetilde\Ff$ in $S^2\wt\times S^2$, and the result of their gluing.  Similarly, the right two figures show the same thing, except with $\DD_\frak a$ replacing $\DD_{\widetilde\Ff}$.  The equivalence of the middle arrow is the content of Figure~\ref{fig:gluck_moves}.}
	\label{fig:gluck_sequence}
\end{figure}

	First, we will show that the union $\DD^\circ\cup\DD_\Ff$ is equivalent to the union $\DD^\circ\cup\DD_\text{caps}$, as illustrated by Figure~\ref{fig:identity_sequence}.  To do this, we must justify that the middle arrow relating the two objects on the right of this figure comes from a pair of destabilizations.  This justification is provided by the sequence of five moves given in Figure~\ref{fig:identity_moves}. First, slide the dark blue and dark green curves over their lighter shaded companions.  Second, we destabilize the dark red curve using the dark blue and dark green curves, which can be assumed to be parallel outside of the local picture, because $\DD^\circ$ is annular.  Third, we slide the dark red curve over the pink curve. Fourth, we destabilize the dark blue curve using the dark red curve and the light green curve, which can be assumed to be parallel away from the local picture, as before. Fifth, we destabilize the dark green curve using the pink and light blue curves, which can be assumed to be parallel outside of the local picture, as before. This completes the claim that the proposed gluing of $S^2\times D^2$ to $X$ coming from $\DD_\Ff$ can be achieved using $\DD_\text{caps}$.

\begin{figure}[h!]
	\centering
	\includegraphics[width=.9\textwidth]{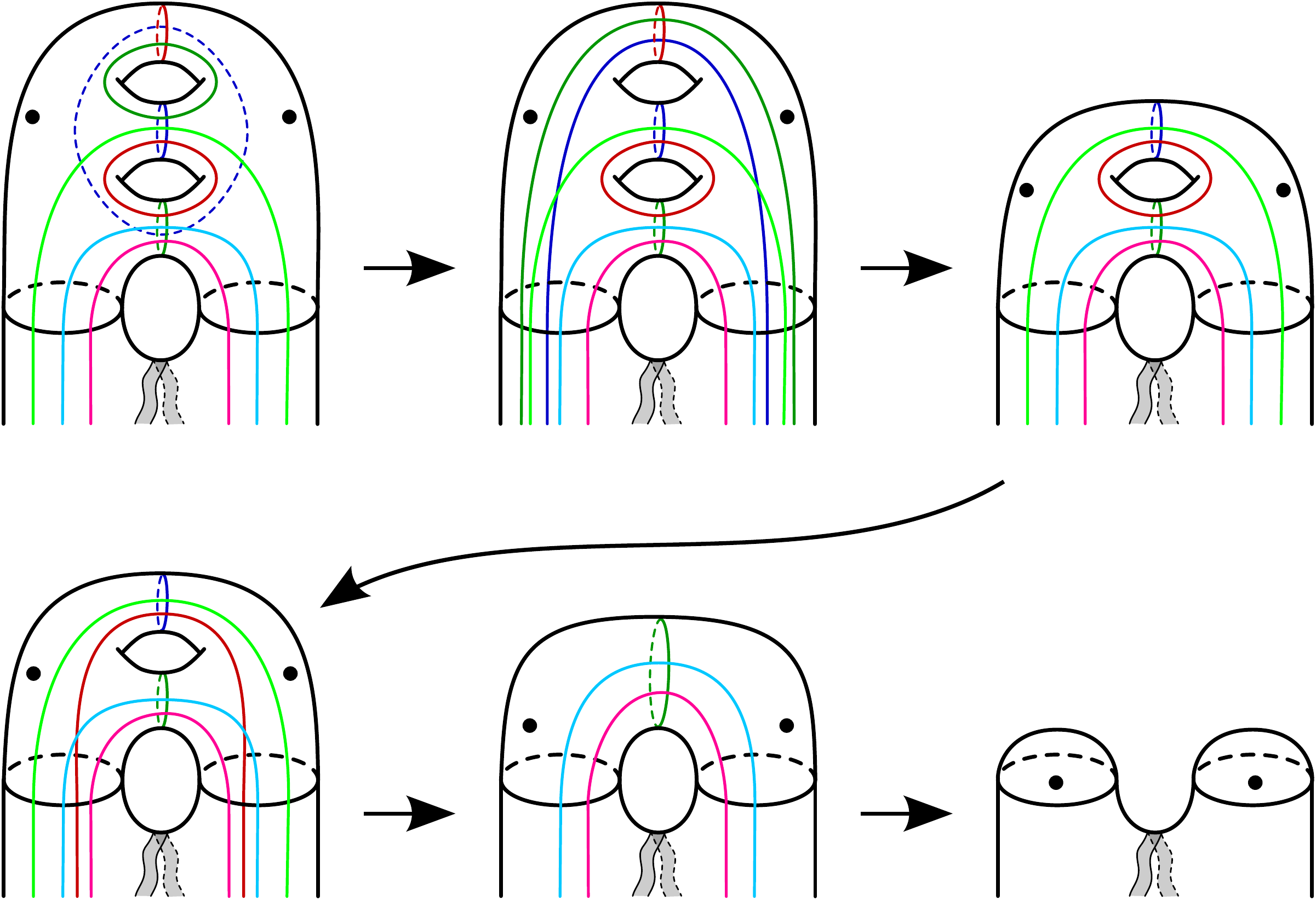}
	\caption{A sequence of handleslides and destabilizations taking $\DD^\circ\cup\DD_\Ff$ to $\DD^\circ\cup\DD_\text{caps}$.  Important is the fact that the original pink, light blue, and light green curves can be assumed, in turn, to be pairwise parallel outside of the local picture. The double points represent the 2--knot $S^2\times\{pt\}$ inside the $S^2\times D^2$ piece described, at first, by $\DD_\Ff$.}
	\label{fig:identity_moves}
\end{figure}
	
	Second, we will show that the union $\DD^\circ\cup\DD_{\widetilde\Ff}$ is equivalent to the union $\DD^\circ\cup\DD_\frak a$, as illustrated by Figure~\ref{fig:gluck_sequence}.  To do this, we must justify that the middle arrow relating the two objects on the right of this figure comes from a pair of destabilizations.  This justification is provided by the sequence of five moves given in Figure~\ref{fig:gluck_moves}. First, we do a left-handed Dehn twist about the left dark blue curve.  Second, we slide the left dark red and dark blue curves over their lighter-shaded companions.  Third, we destabilize the left dark green curve using the parallel pink and light blue curves. Here we are using the fact that the pink and light blue curve can be assumed to be parallel away from the local picture, since they come from the arcs of an annular arced trisection diagram. The result is the lower genus diagram shown in the bottom-left of the figure, where the dark blue and dark red curves running out of the local picture have been lightened.  Fourth, we dark green and dark red curves over their lighter-shaded companions. Fifth, we destabilize the dark blue curve using the parallel dark red and dark green curves.  Here we are using the fact that the light blue and light green curves can be assumed to be parallel away from the local picture, since they come from the arcs in an annular arced trisection diagram.  This completes the claim that the proposed gluing of $S^2\times D^2$ to $X$ coming from $\DD_{\widetilde\Ff}$ can be achieved using $\DD_\frak a$.

\begin{figure}[h!]
	\centering
	\includegraphics[width=.9\textwidth]{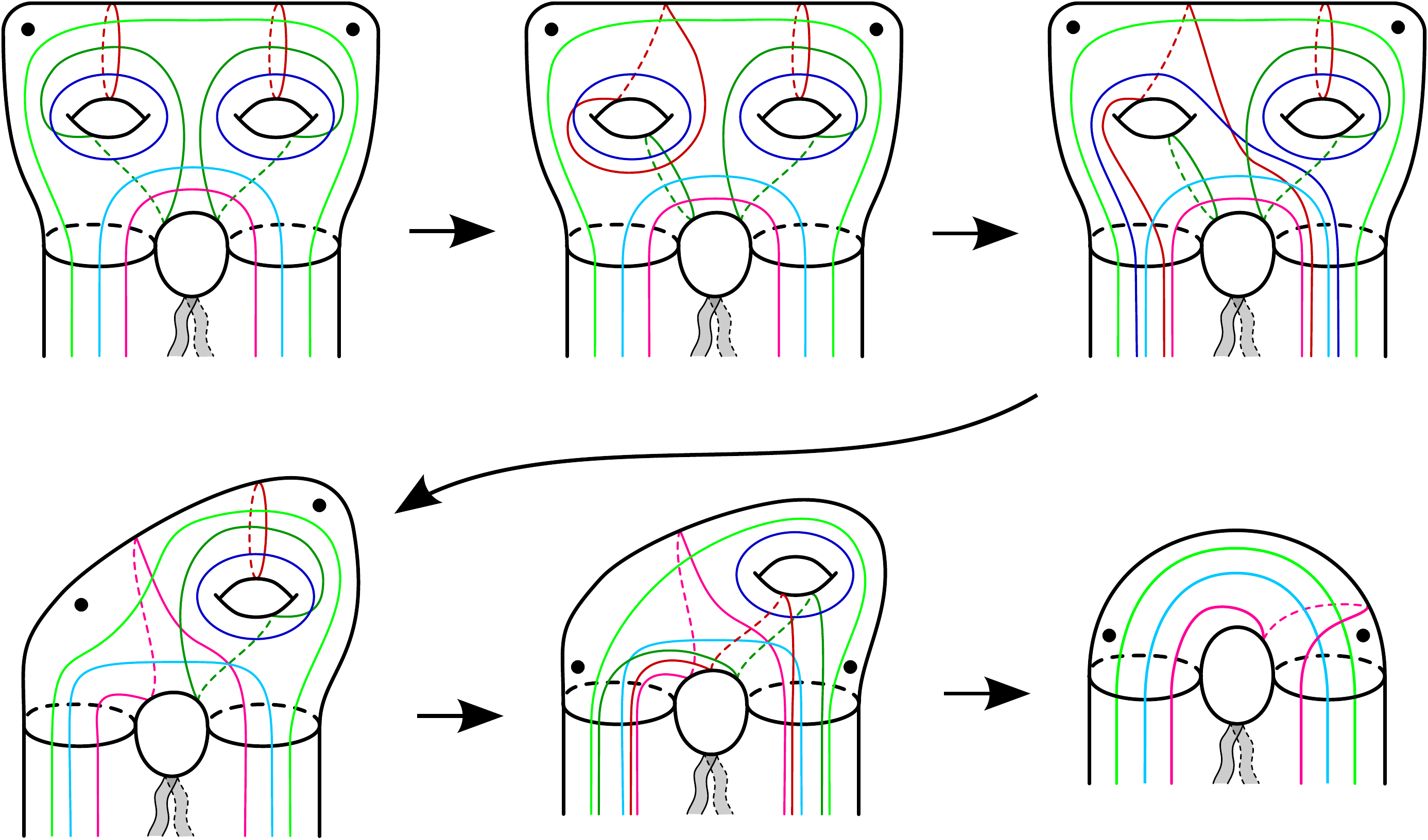}
	\caption{A sequence of slide-diffeomorphisms and destabilizations taking $\DD^\circ\cup\DD_{\widetilde\Ff}$ to $\DD^\circ\cup\DD_\frak a$.  Important is the fact that the original pink curve can be assumed to be parallel, in turn, to the original light blue curve or the original light green curve outside of the local picture.  The double points represent the core $S^2\times\{pt\}$ inside the $S^2\times D^2$ piece described, at first, by $\DD_{\widetilde\Ff}$.}
	\label{fig:gluck_moves}
\end{figure}

	Finally, we claim that the two gluings analyzed above account for both elements of $\Mm(S^2\times S^1)$.  By Lemma~\ref{lem:composition}, we have
$$\Upsilon\left(\DD^\circ,\DD_\Ff\right)\circ\Upsilon\left(\DD_\Ff,\DD_{\widetilde\Ff}\right) = \Upsilon\left(\DD^\circ,\DD_{\widetilde\Ff}\right).$$
We claim that $\Upsilon\left(\DD_\Ff,\DD_{\widetilde\Ff}\right) = \tau$, which implies that $\Upsilon\left(\DD^\circ,\DD_\Ff\right)$ and $\Upsilon\left(\DD^\circ,\DD_{\widetilde\Ff}\right)$ are distinct mapping classes, as desired.  To see that this claim is true, refer to Figure~\ref{fig:gluck_gluing}, which shows the trisection diagram $\DD'' = \DD_\Ff\cup\DD_\Ff$ obtained from gluing $\DD_\Ff$ to $\DD_{\widetilde\Ff}$ in the left two figures.  By the first part of the present lemma, $\DD''$ is equivalent to the third diagram, $\DD_{\widetilde\Ff}\cup\DD_\text{caps}$.  Alternatively, one could invoke the second part of the present lemma to conclude that $\DD''$ is equivalent to $\DD_\frak a\cup\DD_\Ff$.  Either of these diagrams can easily be seen to be equivalent to the fourth diagram of the figure, the genus two trisection diagram for $S^2\wt\times S^2$.  Thus, we have exhibited that $\Upsilon\left(\DD_\Ff,\DD_{\widetilde\Ff}\right) = \tau$, as desired.

\begin{figure}[h!]
	\centering
	\includegraphics[width=.9\textwidth]{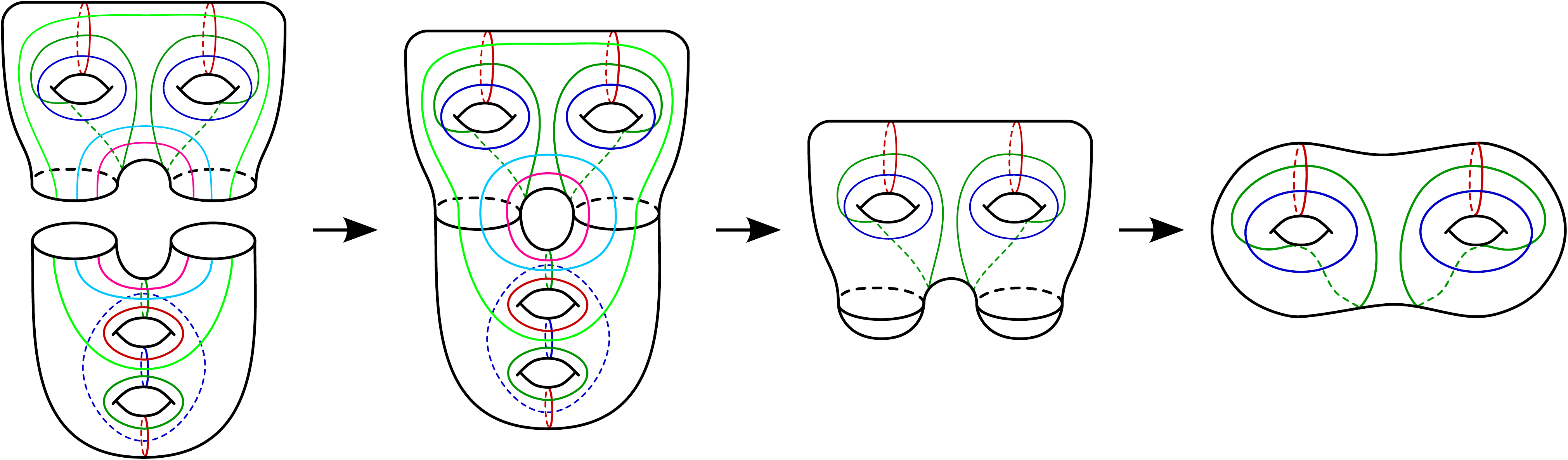}
	\caption{A sequence of diagrams showing that the gluing map $\Upsilon(\DD_\Ff,\DD_{\widetilde\Ff}$ is the Gluck twist $\tau$.}
	\label{fig:gluck_gluing}
\end{figure}

\end{proof}

\begin{remark}
\label{rmk:twisted_gluck}
	The statement of Lemma~\ref{lem:gluck} holds if the role of $\DD_\frak a$ is played by any of the twisted annuli or their mirrors (cf. Remark~\ref{rmk:twisted_blowdown}).
\end{remark}

Finally, we combine the above lemmata to give a proof of Theorem~\ref{thm:surgery}.

\begin{proof}[Proof of Theorem~\ref{thm:surgery}]
	Suppose that $\DD$ is a $(g;k_1,k_2,k_3)$--doubly pointed trisection diagram for a 2--knot $(X,\Kk)$.  If $\Kk$ has self-intersection~$0$, then $\DD^\circ$ satisfies the hypotheses of Lemma~\ref{lem:surgery}, so $\DD^\circ\cup\DD_{B^3\times S^1}$ is a trisection diagram for $X(\Kk)$ of type $(g+1;k_1+1,k_2+1,k_3+1)$.
	
	If $\Kk$ has self-intersection~$\pm 4$, then $\DD^\circ$ satisfies one of the hypotheses of Lemma~\ref{lem:rational}, so $\DD^\circ\cup\DD_{B_{\mp 4}}$ is a trisection diagram for $X_{\pm 4}(\Kk)$ of type $(g+2;k_1+1,k_2+1,k_3+1)$.
	
	If $\Kk$ has self-intersection~$+1$, then $\DD^\circ$ satisfies one of the hypotheses of Lemma~\ref{lem:blowdown}, so $\DD^\circ\cup\DD_\frak a$ is a trisection diagram for $X_{1}(\Kk)$ of type $(g+1;k_1,k_2+1,k_3)$.  Similarly, if $\Kk$ has self-intersection~$-1$, then $\DD^\circ\cup\overline\DD_\frak a$ is a trisection diagram for $X_{-1}(\Kk)$ of type $(g+1;k_1,k_2+1,k_3)$.
	
	If $\Kk$ has self-intersection~$0$, then
	$$\DD^\circ\cup\DD_\text{caps} \text{\ \ \  describes \ \  } E_\Kk\cup_{\Upsilon\left(\DD^\circ,\DD_\Ff\right)}S^2\times D^2,$$
	 and
	$$\DD^\circ\cup\DD_\frak a \text{\ \ \  describes \ \  } E_\Kk\cup_{\Upsilon\left(\DD^\circ,\DD_{\widetilde\Ff}\right)}S^2\times D^2,$$
	by Lemma~\ref{lem:gluck}.  However, by the portion of Lemma~\ref{lem:gluck} shown in Figure~\ref{fig:identity_moves}, $\DD^\circ\cup\DD_\text{caps} = \DD$ is the original trisection diagram for $X$.  It follows that $\Upsilon\left(\DD^\circ, \DD_\Ff\right) = \id$.  Therefore, $\DD^\circ\cup\DD_\frak a$ corresponds to the result of Gluck surgery on $\Kk$ in $X$, as desired.
\end{proof}

\subsection{The dual 2--knot in Gluck surgery}\label{subsec:dual}

In this subsection, we give a slightly more detailed treatment of Gluck surgery; namely, we show how to record information about the dual 2--knot in the Gluck surgery manifold.  The main result is the following.

\begin{reptheorem}{thm:dual}
	Let $\DD^\circ$ be a 0--annular arced diagram for a 4--manifold $E$ with $\partial E\cong S^2\times S^1$.  Then, the diagrams $\DD$ and $\DD'$, as shown in Figure~\ref{fig:twin_diags}, are doubly pointed trisection diagrams for the only 2--knots $(X,\Kk)$ and $(X',\Kk')$ with $E_\Kk\cong E_{\Kk'}\cong E$.
\end{reptheorem}

\begin{proof}
	By Lemma~\ref{lem:gluck}, the underlying trisection diagrams in Figure~\ref{fig:twin_diags} describe the desired ambient 4--manifolds.  The proof of this lemma followed by first gluing on $\DD_\Ff$ and $\DD_{\widetilde\Ff}$, then reducing the result to the unions $\DD^\circ\cup\DD_\text{caps}$ and $\DD^\circ\cup\DD_{\frak a}$, respectively, as illustrated in Figure~\ref{fig:identity_moves} and~\ref{fig:gluck_moves}.  Not relevant at the time, were the double points included in each frame of these figures, which now become relevant.  However, the extra claim here about the cores of these two fillings follow from Lemma~\ref{lem:cores} and by tracing the double points through Figures~\ref{fig:identity_moves} and~\ref{fig:gluck_moves}.
\end{proof}

\section{Examples}\label{sec:examples}

In this section, we apply the techniques and result of this paper to a number of examples.  Many of the 2--knots studied here were first introduced in Subsection~\ref{subsec:examples}.

\begin{example}
\label{ex:unknot}
	Consider the unknot $(S^4, \Uu)$, which has self-intersection zero.  This is the unique 1--bridge 2--knot in $S^4$.   Since this 2--knot is fibered by 3--balls, $E_\Uu \cong B^3\times S^1$.  It follows that the result of sphere surgery on $\Uu$ in $S^4$ is $S^4(\Uu)\cong S^3\times S^1$.  Since the Gluck twist extends over $E_\Uu$, we have that $S^4_*(\Uu)\cong S^4$, as well. Diagrams tracing through these surgery operations are shown in Figure~\ref{fig:unknot}; we see explicitly that the unknot $\Uu$ is its own Gluck-twin.
\end{example}

\begin{figure}[h!]
	\centering
	\includegraphics[width=.9\textwidth]{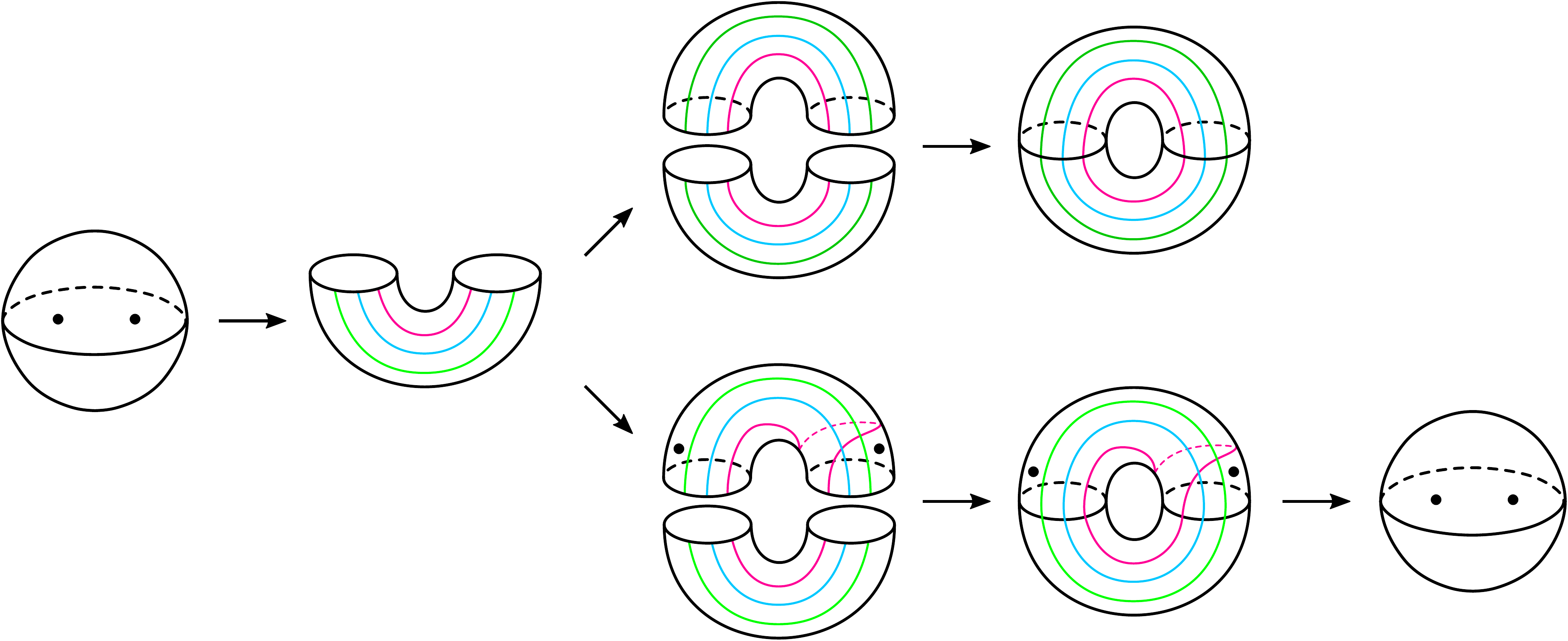}
	\caption{Surgery operations on the unknot $(S^4,\Uu)$.  On the left, we have a doubly pointed, genus zero trisection diagram for this 2--knot, together with the genus zero (annulus) trisection of its exterior $E_\Uu \cong B^3\times S^1$.  The top branch shows the sphere surgery $S^4(\Uu)\cong S^3\times S^1$, while the bottom branch shows the Gluck surgery $S^4_*(\Uu)$.  Note that $(S^4,\Uu)$ is its own Gluck-twin, since the trisection for the Gluck surgery destabilizes.}
	\label{fig:unknot}
\end{figure}

\begin{example}
\label{ex:fibers}
	Consider the fiber 2--knots $(S^2\times S^2, \Ff)$ and $(S^2\wt\times S^2, \widetilde\Ff)$, which are each of self-intersection zero.  doubly pointed, genus two trisection diagrams for these 2--knots, together with genus two trisection diagrams for their exteriors, were given in the fourth and fifth rows of Figure~\ref{fig:1b_Comps}, respectively.  These 2--knots have diffeomorphic exteriors: $E_\Ff \cong E_{\widetilde\Ff}\cong S^2\times D^2$. It follows that they are Gluck-twins, as can be verified by following the bottom branch of Figure~\ref{fig:fibers}.  It also follows, as indicated by the top branch of the figure, that surgery on these 2--knots gives $S^4$.  The belt-circle to these sphere surgeries is the unique curve in $S^4$, giving another way to understand these 2--knots as Gluck-twins.
\end{example}

\begin{figure}[h!]
	\centering
	\includegraphics[width=.9\textwidth]{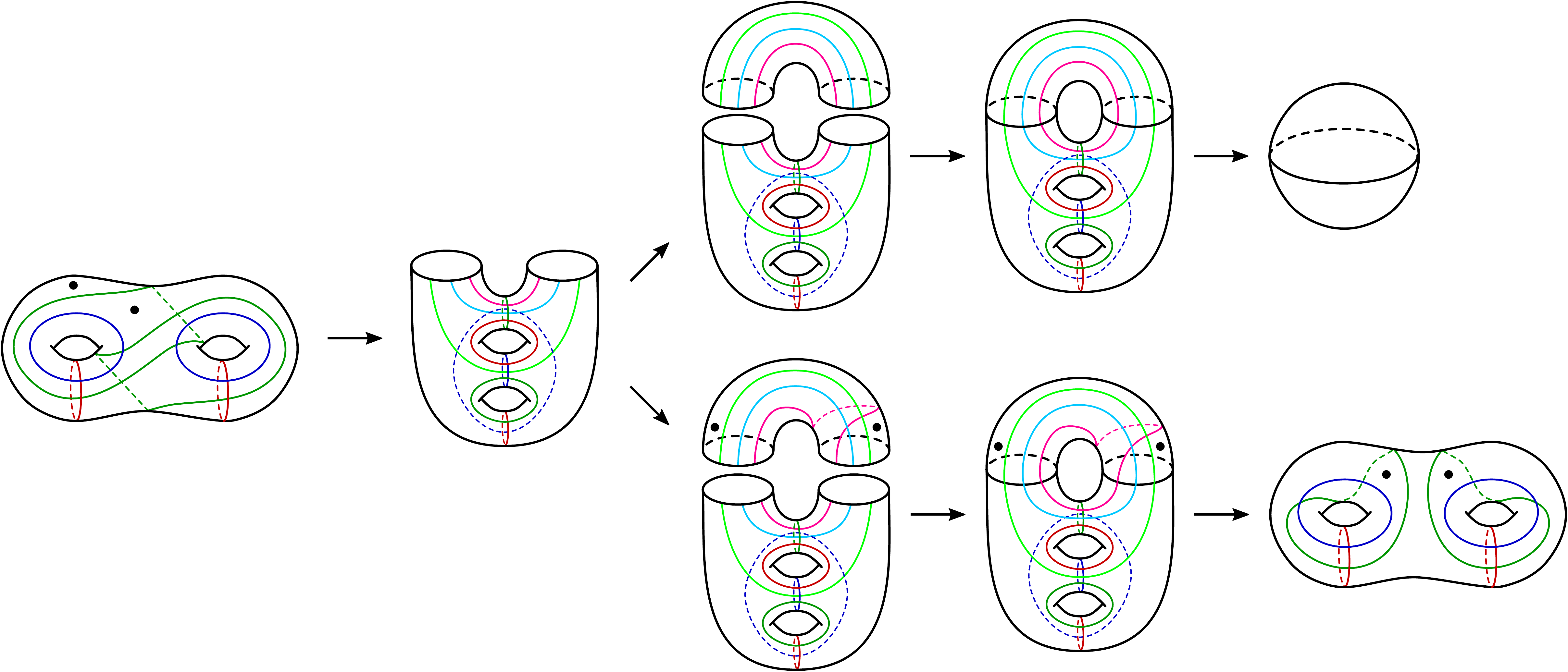}
	\caption{Surgery operations on the fiber 2--knot $(S^2\times S^2,\Ff)$.  On the left, we have a doubly pointed, genus two trisection diagram for this 2--knot, together with the genus two trisection of its exterior $E_\Ff \cong S^2\times D^2$.  The top branch shows the sphere surgery $(S^2\times S^2)(\Ff)\cong S^4$, while the bottom branch shows the Gluck surgery $(S^2\times S^2)_*(\Ff)$, the result of which is the other fiber 2--knot $(S^2\wt\times S^2,\widetilde\Ff)$.}
	\label{fig:fibers}
\end{figure}

\begin{example}
\label{ex:trefoil}
	In~\cite{Mei_Trisections-and-spun_17}, it was shown that the spin of a doubly pointed, genus $h$ Heegaard splitting of a knot $K$ in $S^3$ gives rise to a doubly pointed $(3h,h)$--trisection of the spun 2--knot $(S^3,\Ss(K))$. The simplest example of this construction is when $(S^3,K)$ admits a doubly pointed, genus one Heegaard diagram. Let $T_{p,q}$ denote the $(p,q)$--torus knot, which can be isotoped from its position as a slope on the genus one Heegaard surface to be in 1--bridge position.  See the left column of Figure~\ref{fig:trefoil} and Figure~11 of~\cite{Mei_Trisections-and-spun_17} for examples of doubly pointed Heegaard diagrams when $(p,q)$ is $(2,3)$ and $(3,4)$, respectively.  The right side of each of these figures (top-right in the former case) shows the doubly pointed, genus three trisection diagram for the spin of each torus knot, and the general picture should be clear from these examples.  The bottom-right graphic of Figure~\ref{fig:trefoil} shows the genus four trisection diagram for the result of Gluck surgery on the spun trefoil, which is known to be $S^4$~\cite{Glu_The-embedding-of-two-spheres_62}.
	
	It's not obvious whether or not this genus four trisection diagram for $S^4$ is stabilized; in general, it is unknown whether or not every trisection of $S^4$ with non-zero genus is stabilized.  See~\cite{MeiSchZup_Classification-of-trisections_16} for a complete discussion of the so-called Four-Dimensional Waldhausen Conjecture.  Examples of trisections of $S^4$ that are potential counterexamples to this conjecture described in~\cite{MeiSchZup_Classification-of-trisections_16} and given in~\cite{MeiZup_Characterizing-Dehn_17}.  By considering trisections corresponding to Gluck surgery on spun (or twist-spun) knots, we access a new class of such potential counterexamples; we pose the following question.
	\begin{question}
		Is the trisection diagram constructed by combining the methods of this paper and those of~\cite{Mei_Trisections-and-spun_17} for Gluck surgery on the spin or twist-spin of a non-trivial knot in $S^3$ \emph{\textbf{ever}} stabilized? 
	\end{question}

\end{example}

\begin{figure}[h!]
	\centering
	\includegraphics[width=.9\textwidth]{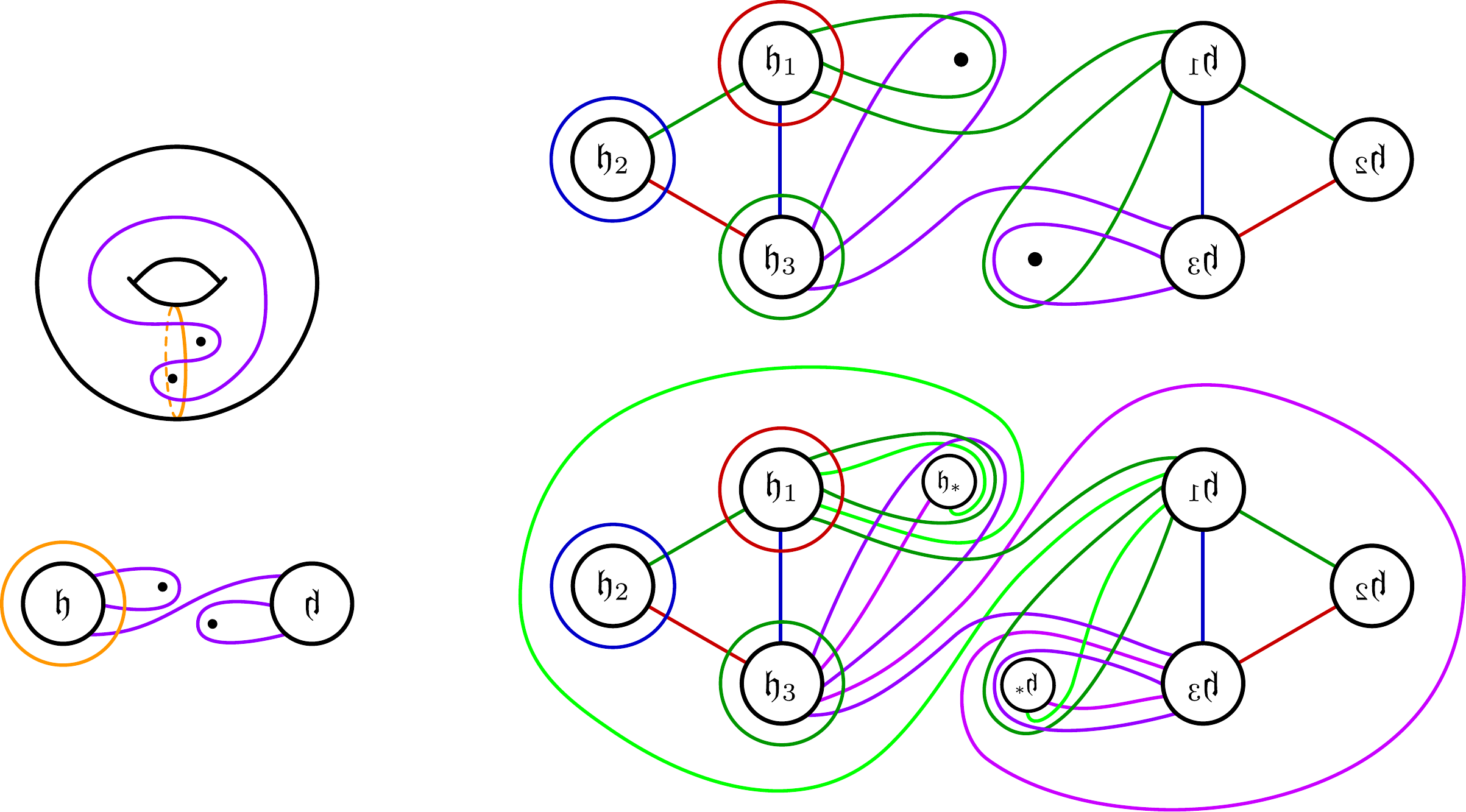}
	\caption{(Left column) Two depictions of the doubly pointed, genus one Heegaard splitting of the right-handed trefoil $(S^3,T_{3,2})$.  (Top-right) A doubly pointed, genus three trisection diagram for the the spun trefoil $(S^4, \Ss(T_{3,2}))$  (Bottom-right) A genus four trisection diagram for the result of Gluck surgery on this 2--knot, which is known to recover $(S^4,\Ss(T_{2,3}))$.  Is this diagram stabilized?}
	\label{fig:trefoil}
\end{figure}

\begin{example}
\label{ex:spin}
	Consider the lens space $L(p,q)$, and let $\Ss_p = \text{Spin}(L(p,q))$ be the spin of this 3--manifold.  We refer the reader to~\cite{Mei_Trisections-and-spun_17} for complete details and a related discussion of these examples.  In particular, it is shown there that $\Ss_p$ admits a genus three trisection.  The top left graphic in Figure~\ref{fig:spin} shows the diagram for $\Ss_5 = \text{Spin}(L(5,2))$ (disregarding the double points for now).  Note that the diffeomorphism type of $\Ss_p$ does not depend on $q$, but the trisection diagram may \emph{a priori}.
	
	Inside $\Ss_p$ are two interesting 2--knots.  The first 2--knot is the core of the spinning construction.  By definition, 
	$$\Ss_p = (L(p,q)^\circ\times S^1)\cup S^2\times D^2,$$
	where $L(p,q)^\circ$ is the punctured lens space. Since we are working with lens spaces, the result is independent of the choice of gluing along the boundary copies of $S^2\times S^1$; the spin and the \emph{twisted spin} are diffeomorphic, in this case~\cite{Plo_Equivariant-intersection_86}.  Let $\Kk_{p,q} = S^2\times\{pt\}$ be the core of the $S^2\times D^2$ filling in the construction.  The second 2--knot is the belt-sphere of a circle surgery on a circle in $S^3\times S^1$. Let $\omega_p$ be the circle such that $[\omega_p] = p\in\Z \cong \pi_1(S^3\times S^1)$.  We claim (and shall soon see) that the circle surgery on $\omega_p$ in $S^3\times S^1$ produces $\Ss_p$, a fact originally due to Pao~\cite{Pao_The-topological-structure_77}.  Let $\Jj_p$ denote the belt-sphere of this circle surgery.
	
	The top left graphic of Figure~\ref{fig:spin} shows a doubly pointed trisection diagram for $(\Ss_p,\Jj_p)$, where the double points have been connected with colored arcs for clarity. The rest of the top row shows the procedure for obtaining a trisection diagram for the result $\Ss_p(\Jj_p)$ of sphere surgery on this 2--knot, which is $S^3\times S^1$, as one can verify by destabilizing the genus four diagram (the top-right graphic) three times. (Note that this proves that $\Ss_p$ is the result of circle surgery on $\omega_p$, as claimed, and that $\Jj_p$ is the belt-sphere of this surgery.)  The bottom row of Figure~\ref{fig:spin} shows the diagrams corresponding to Gluck surgery on $(\Ss_p,\Jj_p)$.  If $p$ is odd, the ambient 4--manifold is still $\Ss_p$ after Gluck surgery.  However, if $p$ is even, a new 4--manifold $\Ss_p'$ is produced that is not homotopy-equivalent to $\Ss_p$~\cite{Pao_The-topological-structure_77}.  
	
	The left graphic of Figure~\ref{fig:spun_bundle} shows a doubly pointed trisection diagram for $(\Ss_p, \Kk_{p,q})$ (again, with clarifying colored arcs added).  The underlying trisection diagram for $\Ss_p$ is obtained from the corresponding diagram discussed in Figure~\ref{fig:spin} by performing one handleslide of each color.  The rest of the figure gives a diagrammatic rendering of sphere surgery on this 2--knot.  By the definition of the spinning construction, $\Ss_p(\Kk_{p,q})\cong L(p,q)\times S^1$.  (Note that the relevance of $q$ persists here.)  Gluck surgery (which we have not depicted) on $(\Ss_p, \Kk_{p,q})$ yields the twisted spin of $L(p,q)$, which, as we remarked above, turns out to be diffeomorphic to the spin~\cite{Plo_Equivariant-intersection_86}. In fact, $\Kk_{p,q}$ is equivalent to its twin, since it has $L(p,q)^\circ$ as a Seifert hyper-surface, and $L(p,q)$ is \emph{untangled}, in the sense of~\cite{Glu_Tangled-manifolds_62}, a sufficient condition for Gluck surgery to preserve the 2--knot.  (See also Section~22 of~\cite{Glu_The-embedding-of-two-spheres_62}.)	
	
\end{example}

\begin{figure}[h!]
	\centering
	\includegraphics[width=.9\textwidth]{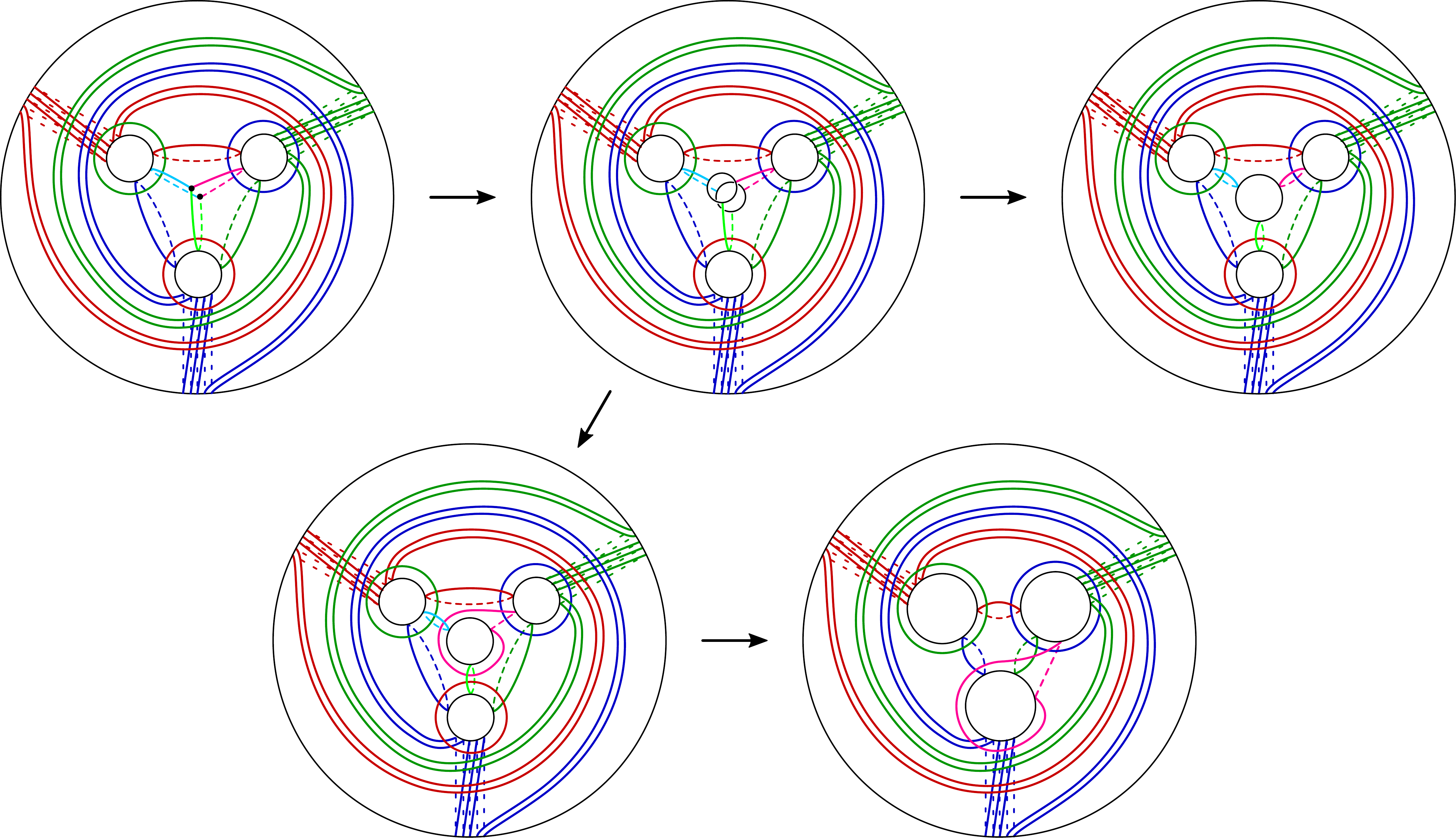}
	\caption{The 2--knot $(\Ss_p,\Jj_p)$, where $\Ss_p$ is the spin of $L(p,q)$ and $\Jj_p$ is the belt-sphere of a circle surgery from $S^3\times S^1$ to $\Ss_p$, together with diagrams depicting the processes of sphere surgery (top row) and Gluck surgery (bottom-row) on this 2--knot.}
	\label{fig:spin}
\end{figure}

\begin{figure}[h!]
	\centering
	\includegraphics[width=.8\textwidth]{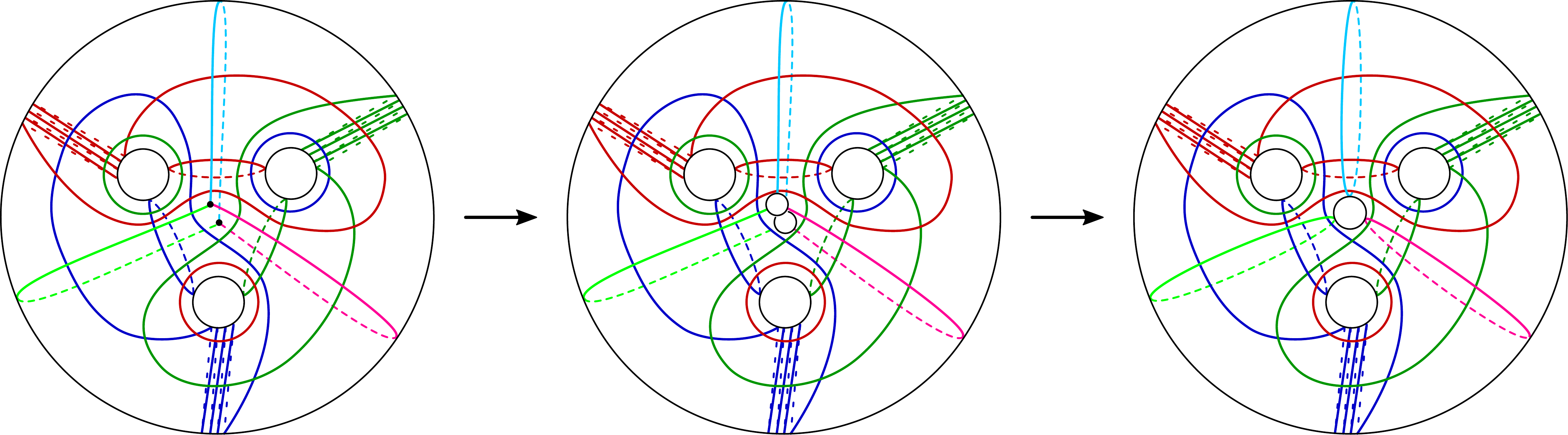}
	\caption{The 2--knot $(\Ss_p, \Kk_{p,q})$, where $\Ss_p$ is the spin of $L(p,q)$, and $\Kk_{p,q}$ is the core of the $S^2\times D^2$ used in the construction of the spin, together with diagrams depicting the process of sphere surgery on this 2--knot.}
	\label{fig:spun_bundle}
\end{figure}

\begin{example}
\label{ex:conic}
	As discussed above, a $(\mp 4)$--rational blowdown is obtained by removing a neighborhood of a 2--knot with self-intersection $\mp 4$ and replacing it with the exterior $B_{\pm4}$ of the degree-two curve $\Cc$ (the \emph{conic}) in $\CP^2$ or its mirror in $\overline\CP^2$, which is a $\Z_2$--homology 4--ball.  Therefore, the simplest example of a $(\pm 4)$--rational blowdown is to perform this operation on the conic itself. In this case, we have
	$$\CP^2_{+4}(\Cc) = B_{+4}\cup_{L(4,1)}B_{-4},$$
	is simply the doubly of $B_{+4}$, hence a $\Z_2$--homology 4--sphere.  Figure~\ref{fig:conic} shows the process of understanding the result of this particular $(+4)$--rational blowdown diagrammatically.  The resulting 4--manifold turns out to be one that shows up in a number of guises, so we conclude with proposition unifying these viewpoints, the proof of which is left to the reader.

\begin{proposition}
\label{prop:Z2}
	The following are all descriptions of the same smooth 4--manifold.
	\begin{enumerate}
		\item The spin $\Ss(\RP^3)$ of the lens space $\RP^3 = L(2,1)$.
		\item One result of circle surgery on the curve in $S^3\times S^1$ representing twice the generator in homology.
		\item The double of an orientable disk-bundle over $\RP^2$ with even Euler number.
		\item The double of the exterior of the conic in $\CP^2$.
		\item The orientable sphere-bundle over $\RP^2$ with even Euler number.
		\item The result of a $(+4)$--rational blowdown of the conic in $\CP^2$.
		\item The two-fold cover of $S^4$, branched along the spin of the Hopf link, whose components are an unknotted sphere and an unknotted torus.
		\item The quotient of $S^2\times S^2$ by the isometry that acts on the first factor by reflection through an equatorial plane and on the second factor by the antipodal map.
		\item Gluck surgery on the fiber of the orientable sphere-bundle over $\RP^2$ with odd Euler number.
	\end{enumerate}
\end{proposition}	

\end{example}

\begin{figure}[h!]
	\centering
	\includegraphics[width=.9\textwidth]{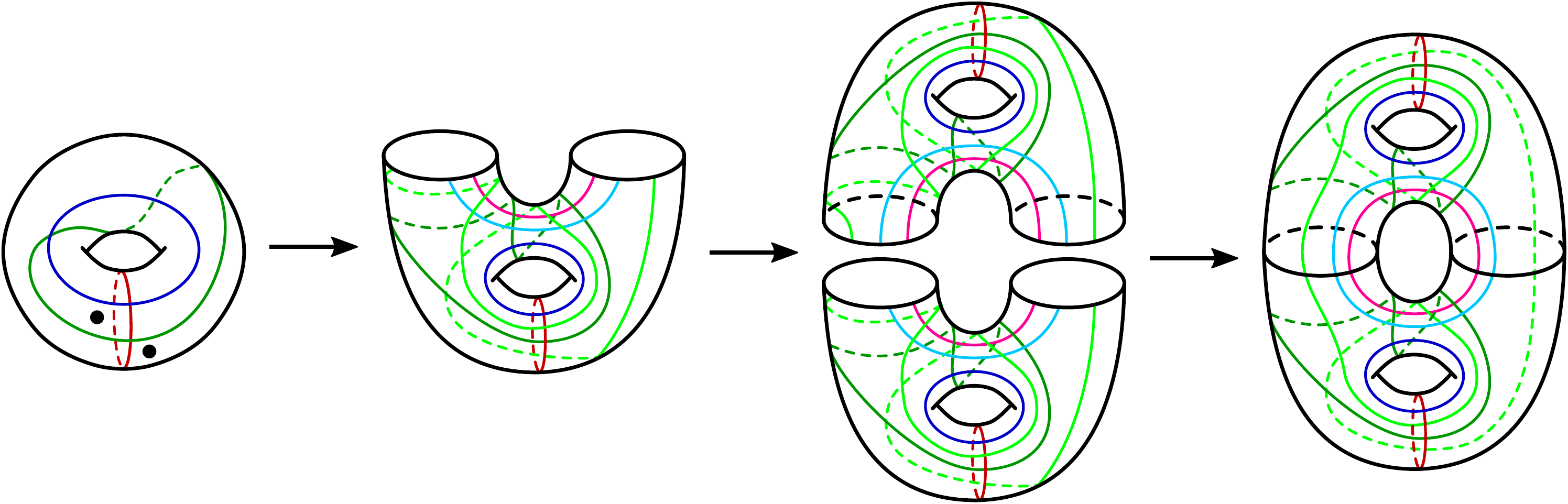}
	\caption{The process of performing a $(+4)$--rational blowdown on the conic $(\CP^2,\Cc)$, stating with a doubly pointed trisection diagram for the 2--knot, passing through a relative trisection diagram for the 2--knot exterior, and terminating with the $\Z_2$--homology 4--sphere $\CP^2_{+4}(\Cc)$.}
	\label{fig:conic}
\end{figure}

\begin{example}
\label{ex:splice}
	The techniques of this paper have another application, which to this point has gone unstated.  Let $(X_1,\Kk_1)$ and $(X_2,\Kk_2)$ be 2--knots, each with self-intersection $p$, and let $E_1$ and $E_2$ denote the corresponding exteriors.  Let $X = E_1\cup_\varepsilon \overline E_2$, where $\varepsilon\colon \partial \overline E_2 \to E_1$ is an orientation reversing diffeomorphism.  This manifold is the \emph{normal connected sum} of $X_1$ and $\overline X_2$ along $\Kk_1$ and $\overline\Kk_2$.  (See Chapter~10 of~\cite{GomSti_4-manifolds-and-Kirby_99} for a discussion of this operation in the symplectic category.)  If $|p|>1$, then the mapping class group of $L(p,1)$ is $\Z_2 = \{\id,\tau\}$, where $\tau$ is the deck transformation corresponding to viewing $L(p,1)$ as the double cover of $S^3$, branched along the (left-handed) torus knot $T_{2,-p}$~\cite{Bon_Diffeotopies-des-espaces_83} .  If $p=1$, there is a unique gluing, since the boundaries are $S^3$. If $p=0$, the boundary is $S^2\times S^1$. Gluck showed that the mapping class group of $S^2\times S^1$ is $\Z_2\oplus\Z_2\oplus\Z_2$; thus, there are \emph{a priori} eight gluings in this case.  If one of the $E_i$ is $B^3\times S^1$, then all eight maps extend, so the gluing is unique.  If one of the $E_i$ is $S^2\times D^2$, then all but one map extend, so our choices are the familiar maps $\id$ and $\tau$ invoked repeatedly above.  However, in general, one expects none of the eight to extend over either of the $E_i$, so all of them might be considered.  Note that four of these mapping classes are orientation reversing. (See Theorem~5.1 of~\cite{Glu_The-embedding-of-two-spheres_62}.)
	
	If $\DD_1$ and $\DD_2$ are doubly pointed trisection diagrams for $(X_i,\Kk_i)$ and $\DD_1^\circ$ and $\DD_2^\circ$ are the corresponding $p$--annular, relative trisection diagrams for the exteriors, the the union $\DD = \DD_1^\circ\cup\overline\DD_2^\circ$ is a trisection diagram for some normal connected sum of $X_i$ and $\overline X_i$.	
\end{example}

\newpage
\bibliographystyle{amsalpha}
\bibliography{Gluck_Bib.bib}

\providecommand{\bysame}{\leavevmode\hbox to3em{\hrulefill}\thinspace}
\providecommand{\MR}{\relax\ifhmode\unskip\space\fi MR }
\providecommand{\MRhref}[2]{%
  \href{http://www.ams.org/mathscinet-getitem?mr=#1}{#2}
}
\providecommand{\href}[2]{#2}
\begin{thebibliography}{CGPC18}

\bibitem[Bon83]{Bon_Diffeotopies-des-espaces_83}
Francis Bonahon, \emph{Diff\'eotopies des espaces lenticulaires}, Topology
  \textbf{22} (1983), no.~3, 305--314. \MR{710104}

\bibitem[CGPC18]{CasGayPin_Diagrams-for-relative_18}
Nickolas~A. Castro, David~T. Gay, and Juanita Pinz\'on-Caicedo, \emph{Diagrams
  for relative trisections}, Pacific J. Math. \textbf{294} (2018), no.~2,
  275--305. \MR{3770114}

\bibitem[CO17]{CasOzb_Trisections-of-4--manifolds_17}
Nickolas~A. Castro and Burak Ozbagci, \emph{Trisections of 4--manifolds via
  {L}efschetz fibrations}, arXiv: 1705.09854, 05 2017.

\bibitem[Coc83]{Coc_Ribbon-knots_83}
Tim Cochran, \emph{Ribbon knots in {$S^{4}$}}, J. London Math. Soc. (2)
  \textbf{28} (1983), no.~3, 563--576. \MR{724727 (85k:57019)}

\bibitem[FS97]{FinSte_Rational-blowdowns_97}
Ronald Fintushel and Ronald~J. Stern, \emph{Rational blowdowns of smooth
  {$4$}-manifolds}, J. Differential Geom. \textbf{46} (1997), no.~2, 181--235.
  \MR{1484044}

\bibitem[GK16]{GayKir_Trisecting-4-manifolds_16}
David Gay and Robion Kirby, \emph{Trisecting 4-manifolds}, Geom. Topol.
  \textbf{20} (2016), no.~6, 3097--3132. \MR{3590351}

\bibitem[Glu62a]{Glu_The-embedding-of-two-spheres_62}
Herman Gluck, \emph{The embedding of two-spheres in the four-sphere}, Trans.
  Amer. Math. Soc. \textbf{104} (1962), 308--333. \MR{0146807}

\bibitem[Glu62b]{Glu_Tangled-manifolds_62}
\bysame, \emph{Tangled manifolds}, Ann. of Math. (2) \textbf{76} (1962),
  62--72. \MR{0151943}

\bibitem[Gor76]{Gor_Knots-in-the-4-sphere_76}
C.~McA. Gordon, \emph{Knots in the {$4$}-sphere}, Comment. Math. Helv.
  \textbf{51} (1976), no.~4, 585--596. \MR{0440561}

\bibitem[GS99]{GomSti_4-manifolds-and-Kirby_99}
Robert~E. Gompf and Andr{{\'a}}s~I. Stipsicz, \emph{{$4$}-manifolds and {K}irby
  calculus}, Graduate Studies in Mathematics, vol.~20, American Mathematical
  Society, Providence, RI, 1999. \MR{1707327 (2000h:57038)}

\bibitem[Joh95]{Joh_Topology-and-combinatorics_95}
Klaus Johannson, \emph{Topology and combinatorics of 3-manifolds}, Lecture
  Notes in Mathematics, vol. 1599, Springer-Verlag, Berlin, 1995. \MR{1439249}

\bibitem[KM18]{KimMil_Trisections-of-surface_18}
S.~{Kim} and M.~{Miller}, \emph{{Trisections of surface complements and the
  Price twist}}, arXiv:1805.00429, 2018.

\bibitem[LP72]{LauPoe_A-note-on-4-dimensional_72}
Fran{\c{c}}ois Laudenbach and Valentin Po{{\'e}}naru, \emph{A note on
  {$4$}-dimensional handlebodies}, Bull. Soc. Math. France \textbf{100} (1972),
  337--344. \MR{0317343 (47 \#5890)}

\bibitem[Mei17]{Mei_Trisections-and-spun_17}
Jeffrey Meier, \emph{Trisections and spun 4--manifolds}, arXiv:1708.01214,
  2017.

\bibitem[MSZ16]{MeiSchZup_Classification-of-trisections_16}
Jeffrey Meier, Trent Schirmer, and Alexander Zupan, \emph{Classification of
  trisections and the {G}eneralized {P}roperty {R} {C}onjecture}, Proc. Amer.
  Math. Soc. \textbf{144} (2016), no.~11, 4983--4997. \MR{3544545}

\bibitem[MZ17a]{MeiZup_Bridge-trisections_}
Jeffrey Meier and Alexander Zupan, \emph{Bridge trisections of knotted surfaces
  in 4--manifolds}, arXiv:1710.01745, 2017.

\bibitem[MZ17b]{MeiZup_Bridge-trisections_17}
\bysame, \emph{Bridge trisections of knotted surfaces in {$S^4$}}, Trans. Amer.
  Math. Soc. \textbf{369} (2017), no.~10, 7343--7386. \MR{3683111}

\bibitem[MZ17c]{MeiZup_Characterizing-Dehn_17}
\bysame, \emph{Characterizing {D}ehn surgeries on links via trisections},
  arXiv:1707.08955, 2017.

\bibitem[Pao77]{Pao_The-topological-structure_77}
Peter~Sie Pao, \emph{The topological structure of {$4$}-manifolds with
  effective torus actions. {I}}, Trans. Amer. Math. Soc. \textbf{227} (1977),
  279--317. \MR{0431231}

\bibitem[Plo86]{Plo_Equivariant-intersection_86}
Steven~P. Plotnick, \emph{Equivariant intersection forms, knots in {$S^4$}, and
  rotations in {$2$}-spheres}, Trans. Amer. Math. Soc. \textbf{296} (1986),
  no.~2, 543--575. \MR{846597}

\bibitem[Sun15]{Sun_Surfaces-in-4-manifolds:_15}
Nathan~S. Sunukjian, \emph{Surfaces in 4-manifolds: concordance, isotopy, and
  surgery}, Int. Math. Res. Not. IMRN (2015), no.~17, 7950--7978. \MR{3404006}

\end{thebibliography}

\end{document}